\documentclass{amsart}
\usepackage{mathtools}
\usepackage{amsthm}
\usepackage{amsmath}
\usepackage{verbatim}
\usepackage{amssymb}
\usepackage{graphicx}
\usepackage{latexsym}
\usepackage{undertilde}
\usepackage[mathscr]{euscript}

\newtheorem{theorem}{Theorem}[subsection]

\newtheorem{lemma}[theorem]{Lemma}
\newtheorem{fact}[theorem]{Fact}

\newtheorem{proposition}[theorem]{Proposition}

\theoremstyle{definition}
\newtheorem{definition}[theorem]{Definition}
\newtheorem{notation}[theorem]{Notation}

\theoremstyle{remark}
\newtheorem{remark}[theorem]{Remark}
\newtheorem{example}[theorem]{Example}

\newcommand{\twon}{2^{\omega} }

\newcommand{\zeroone}{[0,1] } 

\newcommand{\prefix}[2]{{{#1}\reflectbox{1}{#2}}}

\newcommand{\arpi}[2]{\Pi_{#2}^{#1}}
\newcommand{\arsigma}[2]{\Sigma_{#2}^{#1}}
\newcommand{\ardelta}[2]{\Delta_{#2}^{#1}}
\newcommand{\barpi}[2]{\bold \Pi_{#2}^{#1}}
\newcommand{\barsigma}[2]{\bold \Sigma_{#2}^{#1}}
\newcommand{\bardelta}[2]{\bold \Delta_{#2}^{#1}}

\newcommand{\baire}{\omega^\omega}
\newcommand{\fbaire}{\omega^{<\omega}}
\newcommand{\fcantor}{2^{<\omega}}

\newcommand{\Tr}{\text{Tr}}
\newcommand{\PTr}{\text{PTr}}
\newcommand{\cylinder}[1]{\left[ #1 \right]^\prec}

\newcommand{\seqf}[3]{ \left({#1}_{#2}\right)_{{#2}\in{#3}} }
\newcommand{\scheme}[2]{\seqf {#1}{#2}{\baire} }
\newcommand{\opa}[1]{\mathcal A \left\{{#1}\right\}}

\newcommand{\seq}[2]{ \left({#1}_{#2}\right)_{{#2}\in\omega} }
\newcommand{\real}{\mathbb{R}}

\newcommand{\tseq}[2]{\left\{{#1}_\alpha: \alpha<{#2} \right\}}
\newcommand{\cseq}[1]{\tseq{#1}{\mathfrak{c}}}

\newcommand{\gdelta}{G_\delta}
\newcommand{\fsigma}{F_\sigma}

\newcommand{\binw}[1]{\text{bin}_\omega\left({#1}\right)}
\newcommand{\binr}[1]{\text{bin}_{[0,1)}\left({#1}\right)}
\newcommand{\binww}{\text{bin}_\omega}
\newcommand{\binrr}{\text{bin}_{[0,1)}}

\usepackage{hyperref}

\usepackage[UKenglish]{babel}
\begin{document}

\title{Aspects of Classical Descriptive Set Theory}
\author{Alex Galicki\\Supervisors: Prof. David Gauld, Prof. Andr\'e Nies}
\date{\today}
\maketitle

\begin{abstract}
    This report consists of two parts. The first part is a brief exposition of classical descriptive set theory. This part introduces some fundamental concepts, motivations and results from the classical theory and ends with a section on the important result of Addison that established the correspondence between classical and effective notions.

The second part is about invariant descriptive set theory. It consists of a brief introduction to notions relevant to this branch of descriptive set theory and then describes in details some of the relatively recent results concerning equivalence relations between Polish metric spaces.

\end{abstract}

\tableofcontents

\break

\subsection{What is Descriptive Set Theory?}

The most succinct and insightful explanation of what descriptive set theory is known to me is found in Kanamori \cite{Kan95}: ``Descriptive set theory is the definability theory of the continuum''. This theory studies subsets of real numbers, or, more generally, subsets of Polish spaces, that are ``simple'' in some sense: sets with simple topological structure, sets with simple logical description or sets that are simple with respect to some other notion of definability. The main underlying idea is that some difficult questions asked about arbitrary sets, might become easier if asked about sets with simple descriptions.

\subsection{Historical background}
The early history of descriptive set theory is characterised by preoccupation with real numbers. It began, at the turn of the 20-th century, with works by three French analysts - Borel, Lebesgue and Baire - as an investigation into regularity properties of sets of reals. That was the first systematic study of sets of reals. At that time, it wasn't considered a separate discipline, rather it was seen as a natural outgrowth of Cantor's own work. All three of them had reservations regarding what objects are permissible in mathematics. In particular, they were studying the abstract notion of a function, as an arbitrary correspondence between objects, introduced by Riemann and Dirichlet. This led them to investigate some classes of well-behaved functions and well-behaved sets of real numbers. Borel, while working on his theory of measure, introduced the hierarchy of particularly well-behaved sets which are now known as Borel sets. Baire was studying a hierarchy of well-behaved functions, now known as Baire functions. A classical regularity property, the Baire property, is one of consequences of his work. Lebesgue, inspired in part by the work of Borel. introduced what is now known as Lebesgue measurability and this concept, to a degree, subsumed both the Borel hierarchy (through the concept of Lebesgue measurable sets) and the Baire functions (through his concept of a measurable function). Furthermore, Lebesgue established two major results concerning the Borel hierarchy: that the Borel hierarchy is proper and that there exists a Lebesgue measurable subset which is not Borel. A significant part of this paper is devoted to concepts developed in that early part of history of descriptive set theory.

It was the Soviet mathematician Luzin, and two of his early collaborators, Suslin and Sierpinski, who established descriptive set theory as a separate research area. Luzin had studied in Paris where he had become acquainted with the works of the French analysts.
In 1914, at the University of Moscow, Luzin started a seminar, through which he was to establish a prominent school in the theory of functions of a real variable. A major topic of this seminar was the ``descriptive theory of functions''. In 1916 Pavel Aleksandrov, one of the participants in the seminar, established the first important result: all Borel sets have the perfect set property. Soon afterwards, another student of Luzin, Mikhail Suslin, while reading the memoir of Lebesgue, found an error in one of Lebesgue's proofs. Lebesgue claimed that the projection of a Borel subset of the plane is also Borel. Suslin found a counterexample to this assertion. This led to an investigation of what are now known as analytic sets. Suslin initially formulated analytic sets as those resulting from a specific operation (the $\mathcal A-$operation, in Aleksandrov's honor) and later characterised them as projections of Borel sets. Soon after, he established three major results concerning analytic sets:
\begin{enumerate}
\item[i)] every Borel set is analytic;
\item[i)] there exists an analytic set which is not Borel; and
\item[ii)] a set is Borel iff both the set and its complement are analytic.
\end{enumerate}

Arguably, these results began the subject of descriptive set theory. Suslin managed to publish only one paper - he died from typhus in 1919.
In the ensuing years Luzin and Sierpi\'nski formulated the projective sets and the corresponding projective hierarchy. They, and their collaborators in Moscow and Warsaw, managed to prove several important results concerning this new hierarchy. However, this investigation of projective sets, especially with respect to the regularity properties, soon afterwards ran into what Kanamori (\cite{Kan_HI}) describes as ``the total impasse''. In 1938, Kurt G\"odel announced several results related to his work on the constructible hierarchy, which explained the nature of problems descriptive set theorists were facing. It turned out that ZFC is not strong enough to prove the regularity properties for the higher pointclasses in the projective hierarchy. 

A major development in mathematics during 1930's and 1940's was the emergence of recursion theory. One of the early recursion theory pioneers was Stephen Kleene. His work on recursion theory in the 1930's led him to develop a general theory of definability for relations on the integers. In early 1940's he studied the arithmetical relations, obtainable from the recursive relations by application of number quantifiers. Subsequently, he formulated the arithmetical hierarchy and showed some fundamental properties of that hierarchy. Later, he studied the analytical relations, obtainable from the arithmetical ones by application of functions quantifiers, and then formulated and studied the analytical hierarchy that classified these relations. Kleene was developing what would later become an effective descriptive set theory. His student, Addison, in late 1950's discovered a striking correspondence between the hierarchies studied by Kleene and those studied in descriptive set theory: the analytical hierarchy is analogous to the projective hierarchy, and the arithmetical hierarchy is analogous to the first $\omega$ levels of the Borel hierarchy.      

\subsection{Acknowledgements}

The section on historical background is a summary of what can be found in \cite{Kan95} and \cite{Kan_HI}. Most of the historical facts mentioned in this paper are taken from there as well.

With respect to Polish spaces, Borel and projective hierarchies and most of the other classical notions, the paper follows the notation and developments found in \cite{Ker94}. On several occasions, \cite{Jec00},\cite{Kan_HI}, \cite{Mos94} and \cite{Sri98} were used as well.

The section on correspondence between effective and classical descriptive set theories is a brief summary of what can be found in \cite{Kan_HI} and \cite{Mos94}.  

The section on invariant descriptive set theory follows \cite{Gao09} with some results taken from \cite{Bec96},\cite{Mel07} and \cite{Cle01}. The proof of Proposition \ref{t_full} is a slightly modified version of the proof from \cite{Gao09}, which, in turn, is a restatement of the proof from the PhD thesis of J.D. Clemens.

\subsection{Notation index}
\begin{enumerate}
\item[] $int(A)$ - interior of $A$,
\item[] $cl(A)$ - closure of $A$,
\item[] $\sigma(A)$ - the $\sigma-$algebra generated by $A$,
\item[] $\mathcal{N}$ -  the Baire space, that is $\baire$,
\item[] $\mathcal{C}$ -  the Cantor space, that is $\twon$,
\item[] $\cylinder a$ - a cylinder (see Definition \ref{def_trees}),
\item[] $[T]$ - the set of branches of $T$ (see definition \ref{def_trees}),
\item[] $s^\frown t$ - string concatenation (see section \ref{trees}),
\item[] $A\bigtriangleup B$ - symmetric difference, i.e. $(A-B)\cup (B-A)$,
\item[] $\gdelta-$set - a countable intersection of open sets (see (\ref{delta_notation})),
\item[] $\fsigma-$set - a countable union of closed sets (see (\ref{sigma_notation})),
\item[] $F(X)$ - the set of closed subsets of $X$,
\item[] $\barpi0\alpha, \barsigma0\alpha,\bardelta0\alpha$ - Borel pointclasses (see definition \ref{borel_hierarchy}),
\item[] $\barpi1\alpha, \barsigma1\alpha,\bardelta1\alpha$ - projective pointclasses (see definition \ref{projective_hd}),
\item[] $\arpi0\alpha, \arsigma0\alpha,\ardelta0\alpha$ - arithmetical pointclasses (see definition \ref{arithmetical_hd}),
\item[] $\arpi1\alpha, \arsigma1\alpha,\ardelta1\alpha$ - analytical pointclasses (see definition \ref{analytical_hd}),
\item[] $\mathcal A^2$ - second-order arithmetic (see subsection $\ref{effective_subsection}$),
\item[] $\opa{\scheme As}$ - Suslin operation (see definition \ref{suslin_op_def}),
\item[] $\le_W$ - Wadge reducibility (see definition \ref{wadge_def}),
\item[] $\le_B,\sim_B$ - Borel reducibility (see definition \ref{borel_reducibility_d}),
\item[] $G_x$ - stabilizer of $x$ (see definition \ref{orbits_and_stabilizers_d}),

\item[] $\mathbb U$ - the Urysohn space (see subsection \ref{urysohn}),
\item[] $\mathbb X$ - a Borel space of all Polish metric spaces (see subsection \ref{encoding1}),

\item[] $\mathcal X$ - a space of all Polish metric spaces (see subsection \ref{space_of_polish}),

\item[] $\mathfrak c$ - the cardinality of the continuum,
\item[] $\omega$ - the set of natural numbers, 
\item[] $\omega_0,\omega_1\dots$ - order-types of infinite cardinals,

\item[] $d_{X,\omega}$ - metrics used to establish a correspondence between $\mathcal X$ and $\mathbb X$ (see subsection \ref{encoding1}),

\end{enumerate}

\section{Regularity properties of sets of reals}

A \textit{regularity property} is a property indicative of well-behaved sets of reals \cite{Kan95}.   Historically, there are three classical regularity properties:
\begin{enumerate}
\item the Baire property,
\item Lebesgue measurability and 
\item the perfect set property.
\end{enumerate}

For all the above properties it is relatively easy (using a version of the Axiom of Choice) to find a counterexample - a set of reals that does not satisfy some or all of them. Below we define two of the most known such counterexamples.

\paragraph{Bernstein sets} ~

We call a set of real numbers $B$ a \textit{Bernstein set} if for any closed uncountable $C\subseteq\real$ both $C\cap B$ and $C\backslash B$ are non-empty. Let's show that there is such a set.

The cardinality of the set of uncountable closed subsets of $\real$ is $\mathfrak{c}$. Let $\left\{F_\alpha: \alpha<\mathfrak{c} \right\}$ be an enumeration of uncountable closed subsets of $\real$. 

By transfinite induction we define two sequences: $\cseq{x}$ and $\cseq{y}$.
For $\beta< \mathfrak{c}$ define $D_\beta=\tseq{x}{\beta}\cup\tseq{y}{\beta}$. Note that $|D_\beta|<|F_\beta|=\mathfrak c$ and choose $x_\beta, y_\beta\in F_\beta-D_\beta$ with $x_\beta\neq y_\beta$. 

Define $B=\cseq{x}$, it is easy to see that it is a Bernstein set.

\paragraph{Vitali sets} ~

Define an equivalence relation on the real numbers in the interval $[0,1]$: $x\backsim y \iff |x-y|\in\mathbb Q $. Using the Axiom of Choice we can define a set of representatives of all equivalence classes. Such a set is called a \textit{Vitali set}.

\paragraph{} Bernstein sets do not have any of the mentioned regularity properties. Vitali sets do not have the Baire property and are not Lebesgue measurable, but it is possible to find a Vitali set with the perfect set property.

The above mentioned facts will be proven in the subsequent parts of the paper.

\subsection{The Baire property}
                                                                                            
\begin{definition} Let $A$ be a subset of some topological space.

We say that $A$ is \textit{nowhere dense} $\iff$ $int(cl(A))=\emptyset$.

We say that $A$ is \textit{meager} $\iff$ A is a countable union of nowhere dense sets.

We say that $A$ has the \textit{Baire property} $\iff$ $A \bigtriangleup O$ is meager for some open set $O$. 

\end{definition}

The Baire property (from now on - the BP) evolved from the work of \mbox{Ren\'e-Louis Baire}. In particular, from considerations related to the Baire Category Theorem. Intuitively, it says that the set  is ``almost open'' (in fact, \textit{almost open} is another name for the BP). And meagerness is one of the accepted notions of ``smallness'' for sets.

Here is a couple of basic consequences of the above definitions.

\begin{proposition} Let $X$ be a topological space.
\begin{enumerate}
\item The collection of all meager subsets of $X$ forms a $\sigma-$ideal.

\item The collection of all subsets of X having the BP forms the smallest $\sigma-$algebra containing all open sets and all meager sets.
\end{enumerate}
\end{proposition}

\begin{proposition}
Let $X$ be a topological space and $A\subseteq X$, then the following are equivalent:
\begin{enumerate}\label{bp_prop_1}
\item $A$ has the BP;
\item $A=G\cup M$, where $G$ is $G_\delta$ and $M$ is meager;
\item $A=F\backslash M$, where $F$ is $F_\sigma$ and $M$ is meager.
\end{enumerate}
\end{proposition}

\begin{notation}
Let $A\subseteq\real^n$ and $x\in\real.$ The \emph{translation of }$A$ \emph{by} $x$ is defined to be the set
\begin{align*}
 x\oplus A =\left\{ x+a:a\in A \right\}.
\end{align*}

\end{notation}

The following proposition establishes an important fact about meager sets and the BP: their \emph{translation invariance}.

\begin{proposition}\label{bp_translation_invariance} Let $x\in\real$. 
\begin{enumerate}
\item If $M\subseteq\real$ is meager, then $x\oplus M$ is meager too.
\item If $B\subseteq\real$ has the BP, then $x\oplus B$ has the BP too.
\end{enumerate}
\end{proposition}

\noindent Let's show that both Vitali sets and Bernstein sets do not have the BP.

\begin{example}A Bernstein set does not have the BP.
\begin{proof}

Let $A\subseteq\real$ be a Bernstein set. For the sake of finding a contradiction assume that $A$ has the BP. Either $A$ or its complement is not meager. Since the complement of $A$ is a Bernstein set too, WLOG assume that $A$ is not meager. Then by \ref{bp_prop_1}(2) there exists $\gdelta$ set $X\subseteq A$ such that $A\backslash X$ is meager. Furthermore, $X$ must be uncountable and hence by \ref{borel_fact_1} must contain a perfect subset $P$, which gives us a contradiction since $P$ must be an uncountable closed set and a Bernstein set does not contain any uncountable closed subsets.

\end{proof}
\end{example}

\begin{example} A Vitali set does not have the BP.
\begin{proof}
Let $\mathcal V$ be a Vitali set and suppose it has the BP. 

Then there exists an open interval $(a,b)$ such that $(a,b)\backslash\mathcal V$ is meager. This follows from the definition of the BP and from the fact that every open subset of $\real$ is a countable union of disjoint open intervals. 
Furthermore, for any $q\in\mathbb Q$, we have
$$(a,b)\cap(q\oplus \mathcal V)\subseteq((a,b)\backslash\mathcal V)\cap (q\oplus\mathcal V))\subseteq (a,b)\backslash\mathcal V$$ and we get that $(a,b)\cap(q\oplus \mathcal V)$ is meager too. Applying \ref{bp_translation_invariance}(1) we get that $$-q\oplus\left((a,b)\cap(q\oplus \mathcal V)\right)=(a-q,b-q)\cap\mathcal V$$ is meager for any $q\in\mathbb Q$ too. It follows that $\mathcal V=\bigcup_{q\in\mathbb Q}\left[(a-q,b-q)\cap\mathcal V\right]$ is meager and $\real=\bigcup_{q\in\mathbb Q}\left[q\oplus\mathcal V\right]$ is meager too. This is a contradiction.

\end{proof}
\end{example}

\subsection{Lebesgue measurability}

One of the possible ways to define the Lebesgue measure is via the \textit{outer measure} $\mu^*$. Let $X\subseteq\real^n$ and define 
\begin{align*}
\mu^*(X)=\inf\left\{\sum_{i\in\omega}v(I_i):\text{ all $I_i$ are n-dimensional intervals and } X\subseteq \bigcup_{i\in \omega} I_i\right\}
\end{align*}
where $v(I)$ denotes the volume of $I$. A set $X$ is a \textit{null-set} if $\mu^*(X)=0$.

\begin{definition}
A set $A\subseteq \real^n$ is said to be Lebesgue measurable if for every $X\subseteq \real^n$
\begin{align}
\mu^*(X)=\mu^*(X\cap A)+\mu^*(X\backslash A)
\end{align}
\end{definition}

\noindent If $X$ is Lebesgue measurable, then its Lebesgue measure is equal to its outer measure:
\begin{align}
\mu(X)=\mu^*(X).
\end{align}

\begin{proposition}[Basic properties of Lebesgue measurability]
\noindent \begin{enumerate}
\item Every interval is Lebesgue measurable, and its measure is equal to its volume.
\item The Lebesgue measurable sets form a $\sigma-$algebra.
\item Every null set is Lebesgue measurable; null sets form a $\sigma-$ideal and contain all singletons.
\item $\mu$ is $\sigma-$additive and $\sigma-$finite.
\item If $A$ is measurable, then there is a $\fsigma$ set $F$ and a $\gdelta$ set $G$ such that $F\subseteq A \subseteq G$ and $G\backslash F$ is a null set.
\end{enumerate}
\end{proposition}

\noindent An equivalent characterisation of Lebesgue measurability is following:
\begin{fact} 
$X\subseteq \real^n$ is Lebesgue measurable $\iff$ $X\bigtriangleup B$ is a null-set for some Borel set $B$.
\end{fact}

At this point it is clear that Lebesgue measurability resembles the Baire property: where the Baire property indicates that a set is almost-open, the Lebesgue measurability indicates that a set is almost-Borel. A nullset is another commonly accepted notion of ``smallness'' for sets.

There are quite a few similarities between the class of meagre sets and the class of nullsets. Both are $\sigma-$ideals. Both include all countable sets.
Both include some sets of the size of continuum. Both classes have the same cardinality, that of $2^{\mathfrak c}$. Neither class includes an interval. Both classes are translation invariant. The complement of any set of either class is dense. Any set belonging to either class is contained in a Borel set belonging to the same class. 
The following striking result highlights the similarity of those notions.

\begin{theorem}[The Erd\"os-Sierpi\'nski Duality Principle] Assume that the continuum hypothesis holds. Then there exists an involution $f : \mathbb R\rightarrow\mathbb R$ such that for every subset $A$ of $\mathbb R$
\begin{enumerate}
\item  $f(A)$ is meagre if and only if $A$ is a nullset, and 
\item $f(A)$ is null if and only if $A$ is meagre.
\end{enumerate}

\end{theorem}

Just like the BP, Lebesgue measurability is translation invariant. However, the result is a bit stronger: the measure itself is preserved under translation.

\begin{proposition}
Let $A\subseteq\real^n$ be Lebesgue measurable and let $x\in\real$. Then $x\oplus A$ is Lebesgue measurable too and
\begin{align*}
\mu(x\oplus A)=\mu(A).
\end{align*}
\end{proposition}

\begin{remark}
Actually, an even stronger result holds: Lebesgue measure is invariant under isometries.
\end{remark}

There is a number of results, some of which are listed below, relating to ``approximating'' arbitrary subsets of reals with Lebesgue measurable ones and ``approximating'' Lebesgue measurable sets with some ``simple'' sets  (open, closed, compact, cells, et cetera). One way of interpreting those results is that in some precise sense Lebesgue measurable sets are the ones that can be approximated by open, closed and compact sets.

\paragraph{Approximation by open sets}

\begin{proposition} Let $A\subseteq \real^n$ and let $\epsilon>0$. Then there exists an open $O\subseteq \real$ such that $A\subseteq O$ and $\mu(O) \le \mu^\star(A)+\epsilon$.

Hence \begin{align*} \mu^\star(A)=\inf\left\{ \mu(O):A\subseteq O,O \text{ is open} \right\}. \end{align*}
\end{proposition}

\begin{proposition}
A set $A\subseteq \real^n$ is Lebesgue measurable if and only if for every $\epsilon >0$ there is an open set $O$ such that $A\subseteq O$ and $\mu^*(O\backslash A)<\epsilon$.
\end{proposition}

\paragraph{Approximation by closed sets}

\begin{proposition}Let $A\subseteq \real^n$ and let $\epsilon>0$. Then there exists a closed set $F\subseteq \real^n$ such that $F\subseteq A$ and $\mu(A) \le \mu (F)+\epsilon$.

Hence \begin{align*} \mu(A)=\sup\left\{ \mu(F):F\subseteq A,F \text{ is closed} \right\}. \end{align*}
\end{proposition}

\begin{proposition}\label{closed_approx}
A set $A\subseteq \real^n$ is Lebesgue measurable if and only if for every $\epsilon >0$ there is a closed set $F$ such that $F\subseteq A$ and $\mu^*(A\backslash F)<\epsilon$.
\end{proposition}

\paragraph{Approximation by compact sets}

\begin{proposition}\label{approx_compact}
Let $A\subseteq \real^n$ with $\mu^*(A)<\infty$. $A$ is Lebesgue measurable if and only if, for every $\epsilon>0$ there exists a compact set $C$ with $C\subseteq A$ and $\mu^*(A\backslash C)<\epsilon$.
\end{proposition}

\paragraph{Bernstein sets and Vitali sets}~

\noindent To show that these sets are not Lebesgue measurable, first, let's introduce a bit of useful notation.

\begin{notation}
Let $A\subseteq\real$. Its \emph{difference set} is defined to be
\begin{align*}
 A\ominus A =\left\{a-b :a,b\in A \right\}.
\end{align*}

\end{notation}

We will need the following basic fact about difference sets.
\begin{proposition}\label{diff_set_1}
Let $A\subseteq \real$ be Lebesgue measurable with $\mu(A)>0$. Then $A\ominus A$ contains an open neighbourhood of $0$.
\end{proposition}

\noindent Let's show that neither Vitali sets nor Bernstein sets are Lebesgue measurable.

\begin{example}
A Vitali set is not Lebesgue measurable.
\begin{proof} Let $\mathcal V\subseteq\real$ be a Vitali set. 
To find a contradiction, suppose that $\mathcal V$ is Lebesgue measurable.
We need to consider two possibilities:
\begin{enumerate}
\item[(i)] $\mu(\mathcal V)=0$. 
First, note that $$\real=\bigcup_{q\in\mathbb Q}\left(q\oplus\mathcal V\right).$$
To see this, let $x\in\real$. There exists $v_x\in\mathcal V$, such that $x=v_x+q_x$, where $q_x\in \mathbb Q$. Thus $x\in q_x\oplus\mathcal V$. 

Having the above in mind, we get:
$$\mu(\real)=\sum_{q\in\mathbb Q}\mu\left(q\oplus\mathcal V\right)=\sum_{q\in\mathbb Q}\mu\left(\mathcal V\right)=0.$$ Which is not possible.

\item[(ii)] $\mu(\mathcal V)>0.$ From the Proposition \ref{diff_set_1} we know that $\mathcal V\ominus\mathcal V$ contains an open neighbourhood of $0$. Therefore, we can pick a rational non-zero $x\in\mathcal V\ominus\mathcal V$. Then there must be such $v_1,v_2\in\mathcal V$ that $x=v_1-v_2$. But then $v_1\backsim v_2$ and hence $x=0$. A contradiction.
\end{enumerate}
\end{proof}
\end{example}

\begin{example}
A Bernstein set is not Lebesgue measurable.
\begin{proof} Let $A\subseteq\real$ be a Bernstein set. Suppose $A$ is Lebesgue measurable. Then either $\mu(A)>0$ or $\mu(\real\backslash A)>0.$ Without loss of generality, assume $\mu(A)>0$. By \ref{closed_approx} there exists a closed set $C\subseteq A$ such that $\mu(C)>0$, and by \ref{card_of_perfect_sets}(1) there exists a perfect set $P\subseteq C\subseteq A$ such that $\mu(P)>0$. Thus we get a contradiction ($P$ is a closed uncountable set and $C$ by the definition does not contain a closed uncountable set).
\end{proof}
\end{example}

\subsection{The perfect set property}

\begin{definition}
A subset of a topological space is said to be \textit{perfect} if and only if it is nonempty, closed, and has no isolated points.

A subset of a topological space has the \textit{perfect set property} if and only if  it is countable or else has a perfect subset.

\end{definition}

\noindent Both of these concepts originated from Cantor's investigation into the topology of the real line while attempting to prove the Continuum Hypothesis (from now on - CH).
The relevance of these concepts to the CH is seen from the following result proven by Cantor.

\begin{proposition}[Cantor-Bendixson]\label{card_of_perfect_sets}~
\begin{enumerate}
\item For any uncountable closed set $C$ of real numbers, there is a perfect set $P\subseteq C$ with $C\backslash P$ being at most countable.
\item Let $P\subseteq \real$ be perfect. Then $|P|=\mathfrak{c}$.
\end{enumerate}
\end{proposition}

\begin{example}A Bernstein set does not have the perfect set property.
\begin{proof}
Trivial: a Bernstein set is uncountable and contains no closed uncountable subsets.
\end{proof}
\end{example}

On the other hand, it is possible to  find a Vitali set with the perfect set property. The proof, however, is too involved to be included here as an example.

\begin{fact}
There is a Vitali set with the perfect set property.
\begin{proof}
Follows from Theorem 5 in \cite{Mil00}.
\end{proof}
\end{fact}

\section{Polish spaces}
\subsection{Why Polish spaces?}
From now on we will follow the usual development and instead of working with real numbers, we will formulate all the results in the (wider) context of perfect Polish spaces. A Polish space is a separable, completely metrizable space. A perfect Polish space is a Polish space with no isolated points. Of course, $\real^n$ is a perfect Polish space for any $n>0$. 

There is a number of good reasons for switching to Polish spaces. One of them is that the resulting theory is more general and more widely applicable.
Hopefully, this will become more apparent later as the paper progresses. For a start we mention two other reasons.

Descriptive set theory studies sets whose descriptions are ``simple''. Arguably, it is the single most important focus of the theory, and Polish spaces are particularly well suited to study such sets.

It should be noted that Proposition \ref{card_of_perfect_sets} can be extended to Polish spaces: 
\begin{proposition}[Cantor-Bendixson]~
\begin{enumerate}
\item Let $X$ be a Polish space. Then $X$ can be uniquely presented as $X=P\cup C$, where $P$ is a perfect subset of $X$ and $C$ is a countable open set.
\item If $X$ is a perfect Polish space, then $|X|=\mathfrak c$.
\end{enumerate}
\end{proposition}

\noindent Furthermore, with a simple argument it can be shown that the set of distinct Polish topologies itself has the cardinality equal to $\mathfrak c$. This means such objects as relations between Polish spaces and collections of Polish spaces are small enough to be subject to methods of descriptive set theory.

The second reason is the unique relationship between Polish spaces and Borel spaces. An important result, which will be presented and commented upon later, states that all standard Polish spaces share the same Borel structure.

As for the results presented in the previous sections, most of them are easily transferable into the more general setting. The only part for which the translation  into the context of Polish spaces is not trivial is the one related to the Lebesgue measure. In the context of Polish spaces, Lebesgue measurability is subsumed by the notion of \emph{universal measurability}. A set $A\subseteq X$, where $X$ is a standard Borel space, is called \emph{universally measurable} if it is $\mu-$measurable for any $\sigma-$finite measure $\mu$ on $X$. Standard Borel spaces will be introduced later in this paper.

\subsection{Basic properties}
We now proceed with an overview of some of the basic properties of Polish spaces that are relevant in the context of descriptive set theory.

The following properties are simple consequences of the definition.
\begin{proposition}\label{polish_basic}~
\begin{enumerate}
\item The completion of a separable metric space is Polish.
\item A closed subspace of a Polish space is Polish.
\item The product/sum of a sequence of Polish spaces is Polish.
\end{enumerate}
\end{proposition}

\begin{example}
Many topological spaces naturally arising in mathematics are Polish. In particular:
\begin{enumerate}
\item $\real$, $\real^n$, $\mathbb C$, $\mathbb C^n$, $\mathbb I=\zeroone$, $\mathbb I ^n$, $\mathbb I^\omega$ (the Hilbert cube), $\mathbb T=\{x\in\mathbb C: |x|=1\}$, $\mathbb T^n$, $\mathbb T^\omega$ are Polish;
\item Any countable set with the discrete topology, for example $\omega$, is Polish;
\item Any space $A^\omega$, where $A$ is countable set with the discrete topology, is Polish. Two examples of particular significance are $\mathcal N=\baire$ (the Baire space) and $\mathcal C=\twon$ (the Cantor space).
\end{enumerate}
\end{example}

\begin{example}
An open interval $(0,1)$ is also a Polish space, even if its usual metric is not a complete metric. One way to see this is to note that $(0,1)$ is homeomorphic to $\real$ which is Polish. Suppose $\phi:(0,1)\rightarrow \real$ is a homeomorphism. Then we can define the following metric on $(0,1)$:
\begin{align*}
d_\phi(x,y)=\left| \phi(x)-\phi(y)\right|
\end{align*}
It is easy to check that $d_\phi$ is a complete metric.
\end{example}

\subsection{The Baire space and the Cantor space}
\begin{definition}
A topological space is said to be \emph{zero-dimensional} if it is Hausdorff and has a basis consisting of clopen sets. 
\end{definition}

Zero-dimensional Polish spaces play an important role in descriptive set theory. Two most prominent zero-dimensional spaces are $\mathcal C$ and $\mathcal N$, which have nice topological characterisations.

\begin{proposition}[Topological characterisations of $\mathcal N$ and $\mathcal C$]~
\begin{enumerate}
\item The Baire space $\mathcal N$ is the unique, up to a homeomorphism, nonempty Polish zero-dimensional space for which all compact subsets have empty interior.
\item The Cantor space $\mathcal C$ is the unique, up to a homeomorphism, perfect nonempty, compact metrizable, zero-dimensional space.
\end{enumerate}
\end{proposition}

\begin{remark}
$\mathcal N$ is homeomorphic to the subset of irrational real numbers. One way to show this is via continued fractions. There is a one-to-one and onto correspondence between infinite continued fractions and irrational numbers in the $(0,1)$ interval: every infinite continued fraction evaluates to a unique irrational number between 0 and 1. And for every element of $(0,1)\backslash\mathbb Q$ there is a unique infinite continued fraction which evaluates to that number. Every infinite continued fraction can be uniquely represented by a (countably infinite) sequence of positive integers. Let's denote by $\mathcal N_+$ the subset of the Baire space where all coordinates are positive. Clearly, $\mathcal N_+$ represents a space of continued fractions. Let $\phi:(0,1)\backslash\mathbb Q\rightarrow\mathcal N_+$ be the bijective function that maps irrational numbers to (unique) continued fractions. It is not difficult to show that $\phi$ is a homeomorphism. Finally, note that $\mathcal N_+$ is homeomorphic to $\mathcal N$ and $(0,1)\backslash\mathbb Q$ is homeomorphic to the set of irrational numbers.
\end{remark}

\subsection{Trees}\label{trees}

The concept of a tree is a fundamental tool in descriptive set theory.

Let $A$ be a nonempty set and let $n\in\omega$. In what follows we assume that $A$ is given the discrete topology. $A^n$ can be seen either as a set of finite sequences of length $n$ from $A$, or, equivalently, as a set of functions $f:n\rightarrow A$. In this context $A^0=\{\emptyset\}$ and the empty sequence is denoted by $\emptyset$ (it is just an empty set).
 Finally, let $$A^{<\omega}=\bigcup_{n\in\omega}A^n,$$ and let $A^\omega$ be the set of all infinite (countable) sequences from $A$, or, equivalently, functions $f:\omega\rightarrow A$.

Since elements of the above sets are functions, it is possible to use the standard set-theoretic notation when dealing with them. In particular, let $s$ be an element of $A^\omega$ or $A^n$ and let $m\in\omega$ with $m\le n$. Then
\begin{align*}
\prefix sm&=\left(s(0),s(1),\dots s(m-1)  \right) \text{ and }\\
|s|&=\text{ length of }s.
\end{align*} 
Let $a\in A^{<\omega} $ and let $b$ be in either $A^{\omega}$ or $A^{<\omega}$, we say that $a$ is an \emph{initial segment} of $b$, or that $b$ \emph{extends} $a$ if $a\subseteq b$, or, equivalently, if $a=\prefix b{|a|}$. 

For $n\in A$, $t\in A^{<\omega}$ and $s\in A^{\omega} $ we define: 
\begin{align*}
a^\frown n=\left(a(0),\dots,a(|a|-1),n\right),\\
n^\frown a=\left(n,a(0),\dots,a(|a|-1)\right),\\
a^\frown t=\left(a(0),\dots,a(|a|-1),t(0),\dots,t(|t|-1)\right),\\
a^\frown s=\left(a(0),\dots,a(|a|-1),s(0),\dots\right).
\end{align*}

\begin{definition}\label{def_trees}
A \emph{tree} on a set $A$ is a subset $T\subseteq A^{<\omega}$ closed under taking initial segments. We call the elements of $T$ the \emph{nodes} of $T$.

A node $s$ in  $T$ is \emph{terminal} if $s$ has no proper extension in $T$. Otherwise it is \emph{nonterminal} or \emph{intermediate}. 

Let $T$ be a tree and let $s,t\in T$. Then $t$ is called a \emph{successor} of $s$ if $s\subseteq t$ and $|t|=|s|+1$. A tree is a \emph{finitely branching tree} if every node has finitely many successors. 

For a tree $T$ on $\omega$, the set of \emph{branches} of $T$, or the \emph{body} of $T$, is defined as $$[T]=\left\{x\in\baire:\forall n\in\omega~\prefix xn\in T\right\}.$$

$T$ is \emph{well-founded} if $[T]=\emptyset$ and it is \emph{ill-founded} otherwise.

For any $s\in A^{<\omega}$ we let $$\cylinder s=\left\{x\in A^{\omega}:s\subseteq x \right\}.$$
 The collection of all such \emph{cylinder} sets forms the \emph{standard basis} for the (usual product) topology of $A^{\omega}$.
Finally, we call a tree $T$ \emph{pruned} if there are no terminal nodes in it.
\end{definition}

The above definitions are particularly important with respect to $\mathcal N$ and $\mathcal C$, since subsets of $\mathcal N$ can be viewed as trees on $\omega$ and subsets of $\mathcal C$ can be viewed as binary trees.

\begin{proposition}\label{tf_map}
The map $T\mapsto [T]$ is a bijection between pruned trees on $A$ and closed subsets of $A^\omega.$ Its inverse is given by $$F\mapsto T_F=\left\{ \textup{\prefix xn}: x\in F,n\in \omega\right\}.$$
We call $T_F$ the \emph{tree} of $F$.
\end{proposition}

\subsection{Binary expansions} It is often useful to identify natural numbers with some elements of $\fcantor$ and real numbers from the interval $[0,1)$ with some elements of the Cantor space. One common way of doing so is through the binary expansion. 

For the rest of the paper we fix $\binww:\omega\rightarrow\fcantor$ - the usual binary expansion of natural numbers. It is a continuous injection. 

Let $x\in[0,1)$. Define inductively
\begin{align*}
\binr x= 
\begin{cases}
 1^\frown\binr{x-0.5}&\text{when }x\in\left[\frac{1}{2},1\right),\\
 0^\frown\binr{2x}&\text{otherwise.} 
\end{cases}
\end{align*}

The set of \emph{dyadic rationals} is $$\mathbb Q_2=\left\{z2^{-n-1}:z\in\omega,n\in\omega\right\}.$$ 
Elements of the Cantor space can be identified with subsets of $\omega$ by treating (elements of the Cantor space) as characteristic functions of subsets of natural numbers. For $s\in \twon$ define $$N_1(s)=\left\{n\in\omega: s(n)=1\right\}.$$ It is clear that $N_1:\twon\rightarrow\mathcal P(\omega)$ is a bijection. Via this identification it is possible to discuss elements of the Cantor space as if they were subsets of natural numbers. 

\begin{fact}
$\binrr$ is a homeomorphism between $(0,1)\backslash\mathbb Q_2$ and (the subspace of) co-infinite subsets of natural numbers. 
\end{fact}

\begin{example}\label{ex1}
Fix some bijection $g:\omega\times\omega\rightarrow \omega$ and define a function\\ $F_g:\baire\rightarrow \twon$ in the following way. For $s\in\baire$ and $n_0,n_1\in\omega$ let
$$F_g(s)\left(g(n_0,n_1)\right)=\binw{s(n_0)}(n_1).$$

Let's show that $F_g$ is a continuous injection. To prove the continuity, it is sufficient to show that the inverse image of a cylinder set is open. Let $t\in\fcantor$. Then $$F_g^{-1}\left(\cylinder t \right)=\prod_{i\in\omega}T_i$$ where $T_i=\omega$ for all $i>\max\{n_0:\exists n_1\in \omega~g(n_0,n_1)\le n\}$. Thus $T_i=\omega$ for almost all $i$. In the product topology where all coordinate spaces have discrete topology such a set is open.

To show the injectivity, let $x,y\in\baire$ with $x\not= y$. Then for some $n_0\in\omega$ we have $x(n_0)\not=y(n_0)$. Then $F_g\left(x\right)(g(n_0,n_1))\not=F_g\left(y\right)(g(n_0,n_1))$ for some $n_1\in\omega$.
\end{example}

\subsection{Universal Polish spaces}
Several Polish spaces have important universality properties.

\begin{proposition}\label{universal_1}~
\begin{enumerate}
\item Every separable metrizable space is homeomorphic to a subspace of the Hilbert cube.
In particular, Polish spaces are, up to a homeomorphism, exactly the $\gdelta-$subspaces of the Hilbert cube.

\item Every Polish space is homeomorphic to a closed subspace of $\real^\omega$.

\item Every zero-dimensional Polish space is homeomorphic to a closed subspace of $\mathcal N$ and a $\gdelta-$subspace of $\mathcal C$.

\item For every Polish space $X$ there is a closed subset $F\subseteq\mathcal N$ and a continuous bijection $f:F\rightarrow X$. If $X$ is nonempty, there is a continuous surjection $g:\mathcal N\rightarrow X$ extending $f$.

\end{enumerate}
\end{proposition}

One of the consequences of the above proposition and the proposition \ref{tf_map} is that we can view any zero-dimensional Polish space $X$ as $[T]$, where $T$ is  some nonempty pruned tree on $\omega$.

\subsection{Choquet spaces}

\begin{definition}[Choquet games]
Let $X$ be a nonempty topological space. The \emph{Choquet
game }$G_X$ of $X$\ is defined as follows. 

There are two players: player I and player II. Starting with I, both players take turns in choosing nonempty open subsets of $X$ so that their choices form two countable sequences of nonempty open subsets of $X$, $\seq Un$ for player $I$ and $\seq Vn$ for player $II$. Both sequences must satisfy the following condition:

\begin{align}
\text{for all }n\in\omega,~~~~~        U_n\supseteq\ V_n \supseteq U_{n+1}.
\end{align}

Both such defined sequences form a particular \emph{run} of the game. We say that $II$ wins a run of the game if $\bigcap _{n\in\omega}V_n\neq \emptyset$. Otherwise, we say that $I$ wins this run. A strategy for $I$(resp. $II$)\ is a ``rule'' that specifies for a player how to select $U_n$(resp. $V_n$) given all previous selections.
Formally, the tree $T_X$ of all \emph{legal positions} in $G_X$ consists of all finite sequences $(W_0,\dots,W_n)$ of open nonempty subsets of $X$ with $W_0\supseteq W_1\supseteq\dots \supseteq W_n$. A  \emph{strategy} for $I$ in $G_X$ is a subtree $\sigma\subseteq T_X$ such that
\begin{enumerate}
\item $\sigma$ is nonempty;
\item if $(U_0,V_0, \dots, U_n)\in \sigma$, then for all open nonempty $V\subseteq U_n, $ \mbox{$(U_0,V_0, \dots, U_n,V)\in \sigma$};
\item if $(U_0,V_0, \dots, U_n,V_n)\in \sigma$, then for a unique $U \subseteq V_n, $ \mbox{$(U_0,V_0, \dots, U_n,V_n,U)\in \sigma$}.
\end{enumerate} 

A position $W\in T_X$ is \emph{compatible} with $\sigma$ if $W\in \sigma$. A run of the game $\seq Un,\seq Vn$ is \emph{compatible} with $\sigma$ if $(U_0,V_0,\dots,U_n,V_n,\dots)\in[\sigma]$. The strategy $\sigma$ is a \emph{winning strategy} for $I$ if $I$ wins every run compatible with $\sigma$.

The corresponding notions for player $II$ are defined similarly.

\end{definition}

\begin{definition}[Choquet spaces]
A nonempty topological space is a \emph{Choquet space} if player $II$ has a winning strategy in $G_X$.
\end{definition}

\begin{definition}[Strong Chouqet games and spaces] Given a nonempty topological space $X$, the \emph{strong Chouqet game} $G^s_X$ is defined as follows.
Players $I$ and $II$ take turns in choosing nonempty open subsets of $X$ as in the Chouqet game, but additionally $I$ is required to select a point $x_n\in U_n$ and $II$ must choose $V_n\subseteq U_n$ with $x_n\in V_n.$ Conditions for winning are the same as for Chouqet games.

A nonempty space $X$ is a \emph{strong Chouqet space} if player $II$ has a winning strategy in $G^s_X$.
\end{definition}

\begin{fact}
Any strong Chouqet space is Chouqet.
\end{fact}

We are interested in the above notions primarily because of the following result characterising Polish spaces.

\begin{proposition}\label{chouqet_polish}
Let $X$ be nonempty separable metrizable space and $\hat X$ a Polish space in which $X$\ is dense. Then
\begin{enumerate}
\item $X$\ is Chouqet $\iff$ $X$ is comeager in $\hat X$;
\item $X$ is strong Choquet $\iff$ $X$\ is $\gdelta$ in $\hat X$ $\iff$ $X$ is Polish.
\end{enumerate} 
\end{proposition}

\section{Borel sets}

\subsection{Borel hierarchy}

\begin{definition}
Let $X$ be a topological space. A set $A\subseteq X$ is a \emph{Borel set} if it belongs to the smallest $\sigma-$algebra of subsets of $X$ containing all open sets. This $\sigma-$algebra is usually denoted by $\mathcal B(X)$, or by $\mathcal B$ when there is no confusion as to which topological space is assumed.
\end{definition}

Discovered by Borel, these sets form a  hierarchy of particularly well-behaved sets. The first hierarchy for Borel sets, which differs only in minor details to the one presented below, was introduced by Lebesgue.

The starting point is a Polish space - a set with some structure imposed by the topology. The idea is to use this structure and, starting from the ``simplest'' sets - open sets - and using some basic set operations, to form a hierarchy of progressively more complex sets.

\begin{definition}\label{borel_hierarchy}
Let $X$ be a Polish space. For each $\alpha<\omega_1$, let us define the collections $\barsigma{0}{\alpha}(X)$, $\barpi{0}{\alpha}(X)$ and $\bardelta{0}{\alpha}(X)$ of subsets of $X$:
\begin{align*}
\barsigma01(X)=&\text{the collection of all open sets};\\
\barpi01(X)=&\text{the collection of all closed sets};\\
\barsigma0\alpha(X)=&\text{the collection of all sets }A=\cup_{n\in\omega}A_n \text{, where each } A_n\\&\text{belongs to }\barpi0\beta(X) \text{ for some }\beta<\alpha;\\
\barpi0\alpha(X)=&\text{the collection of all sets }A=\cap_{n\in\omega}A_n \text{, where each } A_n\\&\text{belongs to }\barsigma0\beta(X) \text{ for some }\beta<\alpha;\\
\bardelta{0}{\alpha}(X)=&\barsigma{0}{\alpha}(X)\cap \barpi{0}{\alpha}(X).
\end{align*}

Elements of those sequences are called \emph{pointclasses}, and elements of pointclasses are called \emph{pointsets}. When the context is clear, we will write $\barpi{0}\alpha, \barsigma{0}\alpha, \bardelta{0}\alpha $ instead of $\barpi{0}\alpha(X), \barsigma{0}\alpha(X), \bardelta{0}\alpha(X)$.
\end{definition}

\paragraph{Relativity of the Borel hierarchy}~

It is easy to see that the Borel hierarchy is not absolute - it depends on the topology of the underlying Polish space. The following two examples will demonstrate this.

\begin{example}\label{borel_example1}
There are only two clopen sets in $\real$: the $\real$ itself and the empty set $\emptyset$. However, there are infinitely many clopen sets in any zero-dimensional infinite Polish space. Thus the very first level of hierarchy, $\bardelta01$, is very different for $\real$ and for, say, $\mathcal N$.  
\end{example}

\noindent For the second example we'll need the following fundamental fact about Borel sets in Polish spaces.
\begin{proposition}\label{borel_clopen}
Let $(X,\mathcal T)$ be a Polish space and $A\subseteq X$ a Borel set. Then there is a Polish topology $\mathcal T_A\supseteq \mathcal T$ such that $\mathcal B(\mathcal T)=\mathcal B(\mathcal T_A)$ and $A$ is a clopen in $\mathcal T_A$.
\end{proposition}

\begin{example}\label{borel_example2}
Using the previous proposition, it is easy to see that for a given Polish space $(X,\mathcal T)$ and a Borel set $A \in\mathcal B(X)$ there is a Polish space $(X,\mathcal T_A)$ in which $A$ is clopen. Hence $A\in\bardelta01(X,\mathcal T_A)$ irrespectively of where exactly $A$ was located in the Borel hierarchy for $(X,\mathcal T)$.
\end{example}

Similarly to turning Borel sets into clopen sets, we can turn any Borel function into a continuous one.
\begin{proposition}\label{borel_to_continuous}
Let $(X, \mathcal T)$ be a Polish space, $Y$ a second countable space, and $f:X\rightarrow Y$ a Borel function. Then there is a Polish topology $\mathcal T_f\supseteq\mathcal T$ with $\mathcal B(\mathcal T)=\mathcal B(\mathcal T_f)$ such that $f:(X,\mathcal T_f)\rightarrow Y$ is continuous.
\end{proposition}

\paragraph{Hausdorff notation}
Another way of presenting the same hierarchy is the following one. Denote by $G(X)$ the pointclass of open subsets of $X$ and by $F(X)$ the pointclass of closed subsets. For any collection $\mathcal{E}$ of subsets of $X$, define the following operations:
\begin{align}
\label{sigma_notation} \mathcal{E}_\sigma=\left\{\bigcup_{n\in\omega}A_n:A_n\in\mathcal E\right\},\\
\label{delta_notation} \mathcal{E}_\delta=\left\{\bigcap_{n\in\omega}A_n:A_n\in\mathcal E\right\}.
\end{align}

Then we have $\barpi01(X)=F(X)$, $\barsigma01(X)=G(X)$, $\barpi02(X)=G(X)_\delta$(or, in a more familiar way, $\gdelta(X)$), $\barsigma02(X)=F(X)_\sigma $(or $\fsigma(X)$), et cetera. This is an older notation, introduced by Hausdorff. The other one, the now standard $\barsigma0\alpha$ and $\barpi0\alpha$ notation, was introduced by Addison in 1958.

\paragraph{Basic properties of the Borel hierarchy }~

\noindent In the context of descriptive set theory, $X$ is usually assumed to be an uncountable Polish space.

\noindent Here are some basic consequences of the above definition.
\begin{proposition}\label{borel_prop_1} Let $X$ be an uncountable Polish space. Then
\begin{enumerate}
\item $\mathcal B=\bigcup_{\alpha<\omega_1}\barsigma0\alpha=\bigcup_{\alpha<\omega_1}\barpi0\alpha$,\\ hence the hierarchy is called the \emph{Borel hierarchy};
\item for $\alpha<\beta$, 
\begin{align*}
\barsigma0\alpha\subseteq\barsigma0\beta&,
& \barsigma0\alpha\subseteq\barpi0\beta,
\\ \barpi0\alpha\subseteq\barsigma0\beta&,
& \barpi0\alpha\subseteq\barpi0\beta;
\end{align*}
\item for $\alpha>0$, $$\barsigma0\alpha\not\subseteq\barpi0\alpha$$ and $$\barpi0\alpha\not\subseteq\barsigma0\alpha,$$
thus it is a proper hierarchy;
\item All Borel pointclasses are closed under continuous pre-images, finite unions and finite intersections. Moreover, for $\xi\ge 1$, $\barsigma0\xi$ is closed under countable unions, $\barpi0\xi$ is closed under countable intersections and $\bardelta0\xi$ is closed under complements.
\end{enumerate}
\end{proposition}

\begin{remark}
\ref{borel_prop_1}(4) can be made stronger. Suppose $X,Y$ are Polish spaces. Then $\barpi{0}{\alpha}(X)\cup \barpi0\alpha(Y), \barsigma{0}{\alpha}(X)\cup \barsigma0\alpha(Y)$ and $\bardelta{0}{\alpha}(X)\cup \bardelta0\alpha(Y)$ are closed under continuous pre-images. 

This means, for example, that for every some $A\in\barsigma0\alpha(Y)$ and for any continuous function $f:X\rightarrow Y$, $f^{-1}(A)\in\barsigma0\alpha(X).$
\end{remark}

\begin{theorem}[Luzin-Suslin]\label{borel_injective}
Let $X,Y$ be Polish spaces and $f:X\rightarrow Y$ be injective and Borel. Then for any Borel $A$, $f(A)$ is Borel too.
\end{theorem}

\begin{remark}
The above result does not hold for arbitrary Borel functions. 
\end{remark}

\paragraph{Exact location of a set in the Borel hierarchy; Wadge reduction}~
\begin{definition}\label{wadge_def}
Let $X,Y$ be sets and $A\subseteq X,B\subseteq Y$. A \emph{reduction} of $A$ to $B$ is a map $f:X\rightarrow Y$ with $f^{-1}(B)=A$.
If $X,Y$ are topological spaces, we say that $A$ is \emph{Wadge reducible} to $B$, symbolically $A\le_W B$, if there is a continuous reduction     of $A$ to $B$. 
\end{definition}

This can be seen as a notion of relative complexity of sets in topological spaces: if $A\le_W B$, then $A$ is ``simpler'' than $B$. $\le_W$ is called the \emph{Wadge ordering}. We are particularly interested in the Wadge ordering on Borel sets in Polish spaces.

\begin{proposition}     [Wadge's Lemma] Let $S,T$ be nonempty pruned trees on $\omega$, and let $A\subseteq[S]$ and $B\subseteq[T]$ be Borel sets. Then either $A\le_W B$ or $B\le_W [S]\backslash A$.
\end{proposition}

Since zero-dimensional Polish spaces can be seen as nonempty pruned trees, the above proposition can immediately be applied to subsets of zero-dimensional Polish spaces,

For any pointclass $\Gamma$, let its \emph{dual class}, $\widetilde\Gamma$, be the collection of complements of sets in $\Gamma$ (so $\widetilde{\barpi0\alpha}=\barsigma0\alpha$), and say $A$ is \emph{properly} $\Gamma$ if $A \in\Gamma\backslash\widetilde\Gamma$.
If a class is equal to its dual, we call it \emph{self-dual}.
For sets $A,B$ define

$$A\equiv_W\ B \iff A\le_W B~\land~ B\le_W\ A. $$

This is an equivalence relation, whose equivalence classes $[A]_W$ are called \emph{Wadge degrees}. The Wadge ordering imposes a hierarchy, called the \emph{Wadge hierarchy}, on the Borel sets. This hierarchy is much finer than the Borel hierarchy. 

It is easy to see that if $A\in \barsigma0\xi$ (resp., $\barpi0\xi$) and $A\le_W B$, then $B\in \barsigma0\xi$ (resp., $\barpi0\xi$). It follows that $\barpi0\xi$ and $\barsigma0\xi$ are initial segments of the Wadge hierarchy.

\begin{definition}
Let $\Gamma$ be a class of sets in Polish spaces. If $Y$ is a Polish space, we call $A\subseteq Y$ \emph{$\bold\Gamma-$hard} if $B\le_W A$ for any $B\in \Gamma(X),$ where $X$ is a zero-dimensional Polish space. Moreover, if $A\in\Gamma(Y)$, we call $A$ \emph{$\bold\Gamma-$complete}.
\end{definition}

A couple of simple observations regarding the above definition:
\begin{enumerate}
\item If $\Gamma$ is not self-dual on zero-dimensional Polish spaces and is closed under continuous preimages, no $\Gamma-$hard set is in $\widetilde\Gamma$.
\item If $A$ is $\Gamma-$hard ($\Gamma-$complete), then its complement is $\widetilde\Gamma-$hard ($\widetilde\Gamma-$complete).
\item If $A$ is $\Gamma-$hard ($\Gamma-$complete) and $A\le_W B$, then $B$ is $\Gamma-$hard ($\Gamma-$complete if $B$ is also in $\Gamma$).
\end{enumerate}

The above observations provide a basis for a popular method of proving that a given set $B$ is $\Gamma-$hard. To do so, choose some other $\Gamma-$hard set $A$ and show that $A\le_W B$. 

In order to determine the exact location of a set in the Borel hierarchy, one must produce an upper bound, or prove membership in some pointclass $\Gamma$,
and then a lower bound, showing the set is not in the $\widetilde\Gamma$. Usually, the lower bounds are more difficult to obtain, but the notion of a Wadge reduction yields a powerful technique for producing lower bounds. The idea is to find a set $B$ that is known not to belong to $\widetilde\Gamma$, and show that $B\le_W A$. Then $A$ cannot be in $\widetilde\Gamma$ either.

\begin{proposition}
Let $X$ be a zero-dimensional Polish space. Then $A\subseteq X$ is $\barsigma0\xi-$complete if and only if $A\in\barsigma0\xi\backslash\barpi0\xi$. Moreover, a Borel set $A\subseteq X$ is $\barsigma0\xi-$hard if and only if it is not $\barpi0\xi$. Analogous results hold after interchanging $\barsigma0\xi$ and $\barpi0\xi$.
\end{proposition}

For arbitrary Polish spaces the following result is known.

\begin{proposition}
Let $X$ be a Polish space, let $\xi\ge1$ and let $A\subseteq X$.
If $A\in\barsigma0\xi\backslash\barpi0\xi$ then $A$ is $\barsigma0\xi-$complete. Similarly, if $A\in\barpi0\xi\backslash\barsigma0\xi$ then $A$ is $\barpi0\xi-$complete.  
\end{proposition}

The following example demonstrates obtaining an upper bound for a particular set.

\begin{example}[Set of normal sequences is $\barpi{0}{3}$] Let $a\in\twon$. The \emph{density} of $a$ is defined in the following way:
$$\sigma(a)=\lim_{n\rightarrow \omega} \frac{\left|N_1(\prefix an)\right|}{n}.$$

Since we identify elements of $\twon$ with subsets of natural numbers, we can also write about densities of subsets of natural numbers.

A real number $x\in(0,1)$ is said to be \emph{normal in base 2} if the density of its binary expansion is equal to $\frac 12$, i.e. $\sigma(\binr x)=\frac 12.$ Similarly, an element of the Cantor space is \emph{normal} if its density is equal to $\frac12$.

Denote the set of all normal sequences by $N\subseteq\twon$. We claim that $N$ is $\barpi03$.
\begin{proof}

For every $n\in\omega$, $a\in\fcantor$ and $\epsilon\in\mathbb Q\cap(0,1)$, define 
\begin{align*}
\delta_1(a)=\frac{\left|N_1(a)\right|}{n},\\
A_{n,\epsilon}=\left\{y:\in\twon:\left|\delta_1(\prefix yn)-\frac12\right|\le\epsilon\right\}.
\end{align*}

$A_{n,\epsilon}$ is the set of all sequences whose initial segments of length $n$ have densities lying within $\left[\frac12-\epsilon,\frac12+\epsilon\right]$. Then 
$$N=\bigcap_{\epsilon\in \mathbb Q\cap(0,1)}\bigcup_{n\in\omega}\bigcap_{m\ge n}A_{m,\epsilon}.$$

It is easy to see that every $A_{n,\epsilon}$ is a clopen, a finite union of cylinders in fact. A countable intersection of clopens is closed and thus $N$ is $\barpi03$:
\begin{align*}
N=\underbrace{\bigcap_{\epsilon\in \mathbb Q\cap(0,1)}\overbrace{\bigcup_{n\in\omega}\underbrace{\bigcap_{m\ge n}A_{m,\epsilon}}_{\barpi01}}^{\barsigma02}}_{\barpi03}.
\end{align*}

\end{proof}
\end{example}

\begin{remark}
For a given set $X\subseteq\zeroone$, let $D_X$ denote the set of sequences whose densities lie in $X$.
The previous example showed that the set $D_{\left\{\frac12\right\}}$ is $\barpi03$. What is the complexity of $D_X$ for other $X$?  Haseo Ki and Tom Linton proved in \cite{Has94} several general results related to that question. In particular:
\begin{enumerate}
\item for any nonempty $X\subseteq\zeroone$, $D_X$ is $\barpi03-$hard, and
\item for each nonempty $\barpi02$ set $X\subseteq\zeroone$, $D_X$ is $\barpi03-$complete. 
\end{enumerate} 
\end{remark}

\begin{example} Let's show that the range of $F_g$ (see \ref{ex1}) is $\barpi03$.

First note that $s\in F_g\left(\baire\right)$ if and only if 
\begin{align}\label{l1}
s\in \bigcap_{g(n_0,n_1)\in\omega}\bigcup_{N\in\omega}\bigcap_{i\in C_{n_0}(N)}B_i.
\end{align}

Where 
\begin{align*}
C_{n_0}(N)=&\left\{ n\in\omega: n>N \text{ and }\exists n_1~g(n_0,n_1)=n\right\}\text{ and}\\
B_i=&\left\{y\in\twon:y(i)=0\right\}.
\end{align*}

(The idea behind the (\ref{l1}) is that any $n_0-$th element of $x$ can only affect values of a finite number of $F_g(x)(i)$ - those with $i\in C_{n_0}(N)$ for some $N$.)

Since every $B_i$ is a clopen set, it follows that $F_g\left(\baire\right)\in\barpi03.$

\end{example}

\begin{example} For $1\le p\le \infty$ the set $\ell_p $ is a $\barsigma02$ subset of $\real^\omega$. There are two cases to consider. First, suppose $1\le p<\infty$. Then we have
$$\ell_p=\bigcup_{q\in\mathbb Q}K_q$$ where $$K_q=\left\{\seq{x}n\in \real^w:\sum_{n\in\omega}|x_n|^p\le q\right\}.$$
Now all we need to show is that every $K_q$ is a closed set. 

Fix $q\in\mathbb Q$ and let $\seq yn\in \real^w\backslash\ K_q$. Then there exists $N\in\omega $ such that $\sum_{n\le N}|y_n|^p>q$. 

\noindent For an $\epsilon>0$, consider a basic open neighbourhood of $\seq yn$, $$O_{y,\epsilon}=\prod_{n\in\omega}Y_{y,\epsilon},$$ where 
\begin{align*}
Y_{y,\epsilon}=\begin{cases}
(y_n-\epsilon,y_n+\epsilon)& \text{ for }n\le N,\\
\real&\text{ otherwise.}
\end{cases}
\end{align*} 

It is clear, that for a sufficiently small $\epsilon>0$, $O_{y,\epsilon}\subseteq \real^\omega\backslash K_q$. Thus $K_q$ is a closed set.

Now consider the case of $\ell_\infty$. Similarly, we have
$$\ell_\infty=\bigcup_{q\in\mathbb Q}K_{q,\infty},$$ where $$K_{q,\infty}=\left\{\seq{x}n\in \real^w:\sup\{|x_n|\}\le q\right\}.$$
Again, we need to show that every $K_{q,\infty}$ is closed. 

Fix some $q\in\mathbb Q$ and let $y=\seq yn \in \real^\omega\backslash K_{q,\infty}$. We have $\sup_{n\in\omega}|y_n|>q$ and for some $N\in\omega$, $|y_N|>q$. For some $\epsilon>0$, there is an open ball $B(y;\epsilon)$, with $y$ in it, all points of which are greater then $q$. Consider the following basic open set
$$O_{y,\infty}=\prod_{n\in\omega}Y_{n,\infty},$$
where 
\begin{align*} 
Y_{n,\infty}=\begin{cases}
B(y,\epsilon)& \text{ for }n=N,\\
\real& \text{otherwise.}
\end{cases}
\end{align*}
It is clear that $O_{y,\infty}\subseteq \real^\omega\backslash K_{q,\infty}$. Thus $K_{q,\infty}$ is closed.

\end{example}

\begin{example}
Let's show that $c_0$ (the set of all real sequences whose limit is zero) is a $\barpi03$ subset of $\real^\omega$.
We have 
\begin{align*}
c_0=\bigcap_{q\in\mathbb Q}\bigcup_{N\in\omega} B_{q,N},
\end{align*}
where $$B_{q,N}=\left\{\seq yn\in\real^\omega: \forall i\ge N~~|y_i|\le q \right\}.$$

We need to show that every $B_{q,N}$ is closed. To this end, fix $q\in\mathbb Q,N\in \omega$ and let $y=\seq yn\in\real^\omega\backslash B_{q,N}$. Then for some $i\ge N$ we have $|y_i|>q$. Let $\epsilon =\frac{ |y_i|-q}2$ and consider the following open neighbourhood of $y$: $O_y=\prod_{n\in \omega} Y_n,$ where 
\begin{align*}
Y_n=\begin{cases}
B(y_i, \epsilon)& \text{where }n=i,\\
\real &\ \text{otherwise}.
\end{cases}
\end{align*}

It is clear that $O_y\subseteq \real^\omega\backslash B_{q,N}$ and therefore $B_{q,N}$ is closed.
\end{example}

\begin{example}\label{ex2}

It is possible to view a tree $T\subseteq\fbaire$ as an element of $2^{\fbaire}$ by identifying it with its characteristic function.
Let $\Tr\subseteq 2^{\fbaire}$ denote the set of trees and let $\PTr\subseteq2^{\fbaire}$  denote the set of pruned trees. Let's show that if $2^{\fbaire}$ is given the product topology (so that it is homeomorphic to the Cantor space), then $\Tr\in \barpi01$ and $\PTr\in\barpi02$.

First, let's consider the set of trees, $\Tr$. To show that it is closed, it is sufficient to demonstrate that its complement is open. Let $f\in 2^{\fbaire}\backslash \Tr$. Then there exists $s\in\fbaire$ and $t\subseteq s$, such that $f(s)=1$ but $f(t)=0$. The set $$O_{t,s}=\left\{g\in 2^{\fbaire}:g(s)=1 \land g(t)=0\right\}$$
is a basic open set and a neighbourhood of $f$. Furthermore, $O_{t,s}\subseteq 2^{\fbaire}\backslash \Tr$, meaning that the complement of $\Tr$ is open and  $\Tr$ is closed. (Note that any set of form $\left\{g\in2^{\fbaire}: \forall_{0\le i\le N}~ g(x_i)=y_i \right\}$ where $N\in\omega$ and all $x_i,y_i$ are fixed is a basic open set. )

Now let's consider the set of pruned trees, $\PTr$. We have:
\begin{align*}
s\in\PTr\iff \forall a\in\fbaire~~s(a)=1\implies
\begin{cases}
\exists n\in\omega~~s(a^\frown n)=1\text{ and}\\
\forall t\subseteq a ~~ s(t)=1.
\end{cases}
\end{align*}

Equivalently, $$\PTr=\bigcap_{a\in\fbaire}P_a,$$
where 
\begin{align*}
P_a=A_a\cup \bigcup_{n\in\omega} B_{a,n},\\
A_a=\left\{s\in 2^{\fbaire}:s(a)=0\right\},\\
B_{a,n}=\left\{s\in 2^{\fbaire}:s(a)=1,s(a^\frown n)=1, \forall t\subseteq a~~s(t)=1 \right\}.
\end{align*}

Now note, that every $A_a$ is a basic open set, every $B_{a,n}$ is a basic open set, hence every $P_a$ is open and $\PTr$ is $\barpi02$.

\end{example}

\paragraph{Borel subsets of $\real^n$}~

\noindent Since we know that for any $\real^n$ topological space, 
\begin{enumerate}
\item[*]the subsets with the BP form a $\sigma-$algebra,

\item[*]the Lebesgue measurable subsets form a $\sigma-$algebra and
\item[*]all of the above $\sigma-$algebras include open sets,
\end{enumerate}
it follows that all Borel subsets (of $\real^n$) have the BP and all are Lebesgue measurable.

It is not so easy to show that all Borel sets have the perfect set property. The difficulty arises from the fact that a set having the perfect set property does not imply that its complement does, and the sets with the perfect set property do not form a $\sigma-$algebra. This difficulty had been overcome by Pavel Aleksandrov, an early member of Luzin's seminar.

\begin{proposition} \label{borel_fact_1}
For any $n>0$, Borel subsets of $\real^n$ have all three regularity properties.
\end{proposition}

\noindent 

However, since we are working primarily in Polish spaces, we have the following results about Borel subsets of Polish spaces.

\begin{proposition}~ Let $A\subseteq X$ be Borel for some Polish space $X$. Then:
\begin{enumerate} 
\item $A$ has the perfect set property.
\item $A$ has the Baire property.
\item $A$ is universally measurable. 
\end{enumerate}
\end{proposition}

The following proposition introduces the concept of a \emph{universal $\barsigma{0}{\alpha}$ set} and asserts the existence of such sets for uncountable Polish spaces.

\begin{proposition}
Let $X$ be an uncountable Polish space. Then for every $\alpha>0$, there exists a set $U\subseteq X^2$ such that $U$ is $\barsigma{0}{\alpha}$ in $X^2$, and   for every $\barsigma{0}{\alpha}$ set $A$ in $X$ there exists some $a\in X$ such that $$A=\{x:(x,a)\in U\}.$$
$U$ is called a universal $\barsigma{0}{\alpha}$ set.
\end{proposition}

\subsection{Borel spaces}

Let $X$ be a topological space. The measurable space $(X, \mathcal B (X))$ is called the \emph{Borel space} of $X$. 

The following proposition characterizes the Borel spaces of separable metric spaces.

\begin{proposition}
Let $(X,\mathcal S)$ be a measurable space. Then the following are equivalent:
\begin{enumerate}
\item $(X,\mathcal S)$ is isomorphic to some $(Y,\mathcal B(Y))$, where $Y$ is a separable metric space;
\item $(X,\mathcal S)$ is isomorphic to some $(Y,\mathcal B(Y))$, where $Y\subseteq\mathcal C$ (and hence to some subspace of any uncountable Polish space - via \ref{universal_1}(3));
\item $(X,\mathcal S)$ is countably generated and \emph{separates points} (for all distinct points $x,y\in X$, there exists $A\in\mathcal S$ with $x\in A,y\not\in A$).
\end{enumerate}
\end{proposition}

In this context all measurable functions are called \emph{Borel functions}. We call a function $f$ a \emph{Borel isomorphism} if it is a bijection and both $f,f^{-1}$ are Borel functions. \emph{Borel automorphisms} are defined similarly. 

The following important result concerning Borel real-valued functions is due to Lebesgue and Hausdorff.
\begin{proposition}
Let $X$ be a metrizable space. The class of Borel functions $f:X\rightarrow \real$ is the smallest class of functions from $X$ into $\real$ which contains all continuous functions and is closed under taking pointwise limits of sequences of functions.
\end{proposition}

In fact, a more general result, related to the Baire hierarchy, is known. However, the Baire hierarchy is not covered in this paper.

\begin{definition}
A measurable space $(X,\mathcal S)$ is a \emph{standard Borel space} if it is isomorphic to the Borel space $(Y,\mathcal B(Y))$ for some Polish space $Y$.
\end{definition}

It is easy to see that products and sums of sequences of standard Borel spaces are also standard Borel spaces.

\begin{proposition}
Let $X$ be a standard Polish space and let $B\subseteq X$ be Borel. Then the subspace $B$ with the inherited Borel structure is standard Borel.
\end{proposition}

There is a particularly important example of a standard Borel space.
\begin{definition}[Effros Borel space]
Let $(X,\mathcal T)$ be a topological space. Endow the set of closed subsets of $X$, $F(X)$, with the $\sigma-$algebra generated by the collection $A$ of sets of form
\begin{align}\label{effros_generators}
\left\{F\in F(X) : F\cap U\neq\emptyset\right\}
\end{align}
where $U$ varies over open subsets of $X$.

The Borel space $(F(X),\sigma(A))$ is called the \emph{Effros Borel space} of $F(X)$.
\end{definition}

\begin{proposition}
If $X$ is Polish, the Effros Borel space of $F(X)$ is standard.
\end{proposition}

The following result shows a crucial link between Polish spaces and Borel structure of the continuum.  

\begin{proposition}[The Isomorphism Theorem] Let $X,Y$ be standard Borel spaces. Then $X$ and $Y$ are Borel isomorphic if and only if $|X|=|Y|.$
\end{proposition}

In particular, any two uncountable standard Borel spaces are Borel isomorphic and any two perfect Polish spaces share the ``same'' Borel structure.

Despite the fact that perfect Polish spaces are Borel isomorphic, Borel hierarchies are relative and, as we seen in the examples \ref{borel_example1} and \ref{borel_example2}, can be different for different topologies. This can be somewhat inconvenient when working with different topologies simultaneously. However, the fact that all Borel pointclasses are closed under continuous preimages, simplifies this to a degree.

\begin{example}
As in example \ref{ex2}, let's identify trees on $\omega$ with elements of $2^{\fbaire}$. Recall that the set of pruned trees, $\PTr$, is $\barpi02$ in $2^{\fbaire}$, and thus it is a zero-dimensional Polish space (assuming the subspace topology). Let's show that the Effros Borel space of $F(\baire)$ is exactly the one induced by its identification with $\PTr$ via the map $ \mathscr{T} :F\mapsto T_F$.

We need to demonstrate two things: 
\begin{enumerate}
\item[$(\Rightarrow)$] The sets that generate the Effros space of $F(\baire)$ (i.e. the ones defined in (\ref{effros_generators})) are mapped by $ \mathscr{T}$ to some Borel sets in $\PTr$. And
\item[$(\Leftarrow)$] All open sets in $\PTr$ correspond to some Borel sets in the Effros space of $F(\baire)$.  
\end{enumerate}

\begin{proof}[$(\Rightarrow)$]
Let $$S_U=\left\{F\in F(\baire): F\cap U \not=\emptyset\right\}$$ where $U\subseteq\baire$ is open. Since $U$ is open, it is a countable union of some cylinders: $$U=\bigcup_{i\in\omega}\cylinder{u_i},$$
with all $u_i\in \fbaire$.

Let's show that $ \mathscr{T}\left(S_U\right)$ is Borel in $\PTr$. We have 
\begin{align*}
\mathscr{T}\left(S_U\right)=\left\{T\in\PTr: T(U)=1\right\}=\\
\bigcap_{i\in\omega}\bigcap_{u_i\subseteq t}\left\{T\in\PTr:T(t)=1\right\}.
\end{align*}

Since $\left\{T\in\PTr:T(t)=1\right\}$ is a basic open set, $\mathscr{T}\left(S_U\right)$  is $\barpi02$ and thus it is Borel.

\end{proof}

\begin{proof}[$(\Leftarrow)$]
Let $A\subseteq \PTr$ be open. Then $A$ is an intersection of some open set in $2^{\fbaire}$ and $\PTr$. Hence $$A=\bigcup_{i\in \omega}\left\{ T\in\PTr: T(a_i)=1 \right\},$$
where all $a_i\in \fbaire$. We get

\begin{align*}
\mathscr T^{-1}(A)=\bigcup_{i\in\omega}\left\{F\in\ F(\baire): F\cap \cylinder{a_i}\not=\emptyset\right\}.
\end{align*}

Since every $\left\{F\in\ F(\baire): F\cap \cylinder{a_i}\not=\emptyset\right\}$ belongs to the Effros Borel space of $F(\baire)$, $T^{-1}(A)$ belongs to it too.    

\end{proof}

\end{example}

\section{Projective sets}
\subsection{Analytic sets}
Borel sets are not closed under continuous images. Continuous images of Borel sets form a distinct pointclass.

\begin{definition}
Let $X$ be a Polish space. A set $A\subseteq X$ is called \emph{analytic} if there exists a Polish space $Y$ and a continuous function $f:Y\rightarrow X$ with $f(Y)=A$. 
\end{definition}

A simple consequence of the above definition and the Proposition \ref{universal_1}(4) is the following characterisation of analytic sets: a set is analytic if it is a continuous image of the Baire space.
The following proposition provides further basic characterisations of analytic sets.
\begin{proposition}
Let $A\subseteq X$ where $X$ is a Polish space. The following are equivalent:
\begin{enumerate}
\item $A$ is a continuous image of $\mathcal N.$
\item $A$ is a continuous image of a Borel set $B\subseteq Y$ where $Y$ is some Polish space.
\item $A$ is a projection of a Borel set in $X\times Y$, for some Polish space $Y$.
\item  $A$ is a projection of a closed set in $X\times \mathcal N.$
\end{enumerate}
\end{proposition}

\paragraph{Regularity properties of analytic subsets of Polish spaces}~

\noindent Luzin, Suslin and Sierpinski proved that analytic subsets of $\real^n$ have all three regularity properties. Virtually the same result holds for Polish spaces.
\begin{proposition} Let $X$ be a Polish space and let $A\subseteq X$ be analytic. Then
\begin{enumerate}
\item $A$ has the perfect set property,
\item $A$ has the BP,
\item $A$ is universally measurable.
\end{enumerate}
\end{proposition}

\subsection{The Suslin operation $\mathcal A$}

Historically, analytic sets were defined in terms of a special operation, the Suslin operation. This operation remains
an important tool in descriptive set theory.

\begin{definition}\label{suslin_op_def}
Let $\scheme As$ be a collection of sets indexed by elements of $\fbaire$. Define 
\begin{align}
\opa{\scheme As}=\bigcup_{a\in \baire}\bigcap_{n\in\omega}A_{\prefix an}
\end{align}

The operation $\mathcal A$ is called the \emph{Suslin operation}.
\end{definition}

\begin{proposition}
A set $A$ in a Polish space $X$ is analytic if and only if it is the result of the Suslin operation applied to a family of closed sets.
\end{proposition}

\begin{definition}[Luzin scheme]
A \emph{Luzin scheme} on a set $X$ is a family $\seqf As{\fbaire}$ of subsets of $X$ such that
\begin{enumerate}
\item $A_{s^\frown i}\cap A_{s^\frown j}=\emptyset $ for all $s\in\fbaire$ and all $i\neq j,$
\item $A_{s^\frown i}\subseteq A_s$ for all $s\in\fbaire, i\in \omega;$
\end{enumerate}
Furthermore, if $(X,d)$ is a metric space and $\seqf As{\fbaire}$ is a Luzin scheme on $X$, we say that $\seqf As{\fbaire}$ has \emph{vanishing diameter} if for all $x\in \baire$, \mbox{$\lim_n diam(A_{\prefix xn})=0.$}
In this case let $$D=\{x\in\baire: \cap_n A_{\prefix xn}\neq \emptyset \},$$
 and define $f:D\rightarrow X$ by taking $\{f(x)\}=\cap_n A_{\prefix xn}$. We call $f$ the \emph{associated map}.
\end{definition}

\begin{proposition}\label{luzin_scheme_1}
Let $\seqf As{\fbaire}$ be a Luzin scheme on a metric space $(X,d)$ that has a vanishing diameter and let $f:D\rightarrow X$ be the corresponding associate map. Then
\begin{enumerate}
\item $f$ is injective
 and continuous, 
\item if $(X,d)$ is complete and each $A_s$ is closed, then $D$ is closed, \item if each $A_s$ is open, then $f$ is an embedding.
\end{enumerate}
\end{proposition}

\begin{example}
One of the characterisations of Borel sets is the following one: the Borel sets are exactly the continuous injective images of closed subsets of the Baire space. We know (\ref{borel_injective}) that all continuous injective images of Borel sets are Borel too. What is left to show is that for any Borel set $B$, we have $B=f(A)$ where $f$ is continuous, injective on $A$ and $A\subseteq\baire$ is closed.

Let $B\subseteq X$ be a Borel  subset of a Polish space$(X,\mathcal T)$. There exists (via \ref{borel_clopen}) a finer Polish space $(X, \mathcal T_B)$ where $B$ is a clopen. Then, via \ref{polish_basic}(2), $B$, as a subspace of $(X, \mathcal T_B)$, is a Polish space too. What is left is to show that every Polish space is a continuous injective image of a closed subset of $\baire$. This is exactly what \ref{universal_1}(4) states, but for the sake of our example we'll employ \ref{luzin_scheme_1}(1) to prove it.

\begin{proof}
Let $X$ be a Polish space. The idea is to construct a Luzin scheme $\seqf As{\fbaire}$ with a few additional properties, so that the associated map is continuous, bijective and its domain is a closed subset of the Baire space.

The construction of $\seqf As{\fbaire}$ proceeds as follows. Set $A_\emptyset=X$ and note that $A_\emptyset$ is $\barsigma02$. Given $s\in \fbaire$, assume $A_s$ is $\barsigma02$ and let $\delta=2^{-|s|}$. Then $A_s$ can be written as $A_s=\bigcup_{i\in\omega}K_i$ where all $K_i$ are closed. 

Let $C_i=\bigcup_{n=0}^i K_n$ so that we have $A_s=\bigcup_{i\in\omega}C_{i+1}\backslash C_i$. Fix $\seq Un$ - an open covering of $X$ such that $diam(U_n)<\delta$ for every $n\in\omega$. Define $\seq D{n,i}$ and $\seq E{n,i}$ as $$D_{n,i}=U_n\cap\left[C_{i+1}\backslash C_i\right],$$
$$E_{n,i}=D_{n,i}\backslash \bigcup_{0\le j\le n-1}D_{j,i}.$$

Note that $D_{n,i}$ are $\bardelta02$ and $E_{n,i}$ are $\barsigma02$ with $diam\left(E_{n,i}\right)<\delta$. Then we have $C_{i+1}\backslash C_i=\cup _{n\in\omega}E_{n,i}$ and therefore 
\begin{align}\label{l33}
A_s=\bigcup_{n,i\in\omega}E_{n,i}.
\end{align}
 Furthermore, 
\begin{align}\label{l32}
\overline{E_{n,i}}\subseteq \overline{C_{i+1}\backslash C_i}\subseteq C_{i+1}\subseteq A_s.
\end{align}

\noindent Finally, fix a bijection $g:\omega\times\omega\rightarrow \omega$ and for every $n,i\in\omega$, let 
\begin{align}\label{l31}
A_{s^\frown g(n,i)}=E_{n,i}.
\end{align} 

\noindent The Luzin scheme so defined has the following additional properties: 
\begin{enumerate}
\item $A_\emptyset=X,$
\item for all $s\in\fbaire,$ $A_s$ is $\barsigma02$ (due to \ref{l31}),
\item for all $s\in\fbaire,$ $diam(A_s)< 2^{-|s|}$ and so the scheme has a vanishing diameter, 
\item for all $s\in\fbaire,$ $A_s=\bigcup_{i\in\omega}A_{s^\frown i}=\bigcup_{i\in\omega}\overline{A_{s^\frown i}}$ (because of \ref{l31},\ref{l32} and \ref{l33}),
\item the associated map $f:D\rightarrow X$ is 
\begin{enumerate}
\item continuous and injective (since \ref{luzin_scheme_1}(1) is applicable) and
\item surjective (because of 1. and 4.).

\end{enumerate}
\end{enumerate}

The only thing left to demonstrate is that $D$, the domain of the associated map, is a closed subset of the Baire space. To see this, let $\seq an$ be a sequence of points in $D$ converging to some $a\in \baire$. Let $\epsilon>0$. We can choose $m'\in\omega$ such that $\epsilon>2^{-m'}$ and for all $ m>m'$, $\prefix {a_m}{m'}=\prefix a{m'}$ and $f(a_m)\in A_{\prefix a{m'}}$. Since $X$ is Polish and $diam\left(A_{\prefix a{m'}}\right)\le 2^{m'}<\epsilon$, the sequence  $\left(f(a_n)\right)_{n\in\omega}$
is Cauchy and hence converges to some point $x\in X$. 

We know that $x\in \bigcap_{i\in\omega}A_{\prefix a i}=\bigcap_{i\in\omega}\overline{A_{\prefix ai}}$, so it must that $a\in D$ and $f(a)=x$. This means that $D$ is closed.
\end{proof}
\end{example}

\subsection{The Projective Hierarchy}

Starting with analytic sets and using the operation of taking the complement and the projection operation, we get the following hierarchy.
\begin{definition}\label{projective_hd}
Let $X$ be a Polish space and let $n\ge 0$. We define the collections $\barsigma1n,\barpi1n$ and $\bardelta1n$ of subsets of $X$ as follows:
\begin{align*}
\barsigma11=&\text{the collection of all analytic sets,}\\
\barpi11=&\text{the complements of all analytic sets,}\\
\barsigma1{n+1}=&\text{the collection of the projections of all }\barpi1n\text{ sets in }X\times \mathcal N ,\\
\barpi1n=&\text{the complements of the }\barsigma1n \text{ sets in }X,\\
\bardelta1n=&\barsigma1n\cap\barpi1n.
\end{align*}

The sets belonging to one of the above pointclasses are called \emph{projective sets}.
\end{definition}

Suslin established the following basic fact relating the projective hierarchy and the Borel hierarchy.
\begin{proposition}
Let $X$ be a Polish space. $\mathcal B(X)=\barpi11(X)\cap \barsigma11(X)$.
\end{proposition}

\begin{proposition}[Basic properties of the projective hierarchy]\label{projective_prop_1} Let $X$ be a Polish space. Then
\begin{enumerate}
\item for $\alpha<\beta$, 
\begin{align*}
\barsigma1\alpha\subseteq\barsigma1\beta&,
& \barsigma1\alpha\subseteq\barpi1\beta,
\\ \barpi1\alpha\subseteq\barsigma1\beta&,
& \barpi1\alpha\subseteq\barpi1\beta;
\end{align*}
\item for $\alpha>1$, $$\barsigma1\alpha\not\subseteq\barpi1\alpha$$ and $$\barpi1\alpha\not\subseteq\barsigma1\alpha,$$
thus the projective hierarchy is proper;
\item All projective pointclasses are closed under Borel pre-images.
\end{enumerate}
\end{proposition}

\subsection{Regularity properties of the projective pointclasses}

By this point we've established that all three regularity properties hold for all $\bardelta11$ and for all $\barsigma11$ pointsets. It follows, trivially, that all $\barpi11$ pointsets are universally measurable. Finally, the following fact is known.
\begin{proposition}
All $\barpi11$ pointsets have the Baire property.
\end{proposition}

\noindent As it turns out, nothing else can be proven (in ZFC) for the higher projective pointclasses. In 1938 (\cite{God38}), G\"odel announced the following result.
\begin{proposition}
If $V=L$, then:
\begin{enumerate}
\item there is a $\bardelta12$ set of reals which is not Lebesgue measurable, and
\item there is a $\barpi11$ set of reals which does not have the perfect set property.
\end{enumerate}
\end{proposition}

\noindent Regarding the Baire property, the following result is known (according to Kanamori \cite{Kan_HI}, the first known explicit proof of this result appeared in Mycielski \cite{Myc64}).

\begin{proposition}
If $V=L$, then there is a $\bardelta12$ set of reals without the Baire property.
\end{proposition}

\noindent A simple consequence of the above results is that it is not possible to prove in ZFC either of the following statements:
\begin{enumerate}
\item all $\bardelta12$ pointsets are Lebesgue measurable;
\item all $\bardelta12$ pointsets have the Baire property;
\item all $\barpi11$ pointsets have 
he perfect set property.
\end{enumerate}

\noindent It is still possible to prove results concerning the regularity properties of the higher pointclasses of the projective hierarchy, however, in Dana Scott's words, ``if you want more you have to assume more''. That is, other, stronger assumptions are needed and thus most such results have a strong metamathematical flavour.

\section{The classical-effective correspondence}\label{effective_subsection}

The subject of effective descriptive set theory is vast and lies outside of scope of the present paper. However, the correspondence between classical and effective theories, discovered by Addison, is too important to be left out. In this section we define effective hierarchies and present Addison's result. For this section alone we assume some rudimentary familiarity with logical and computability notions.

\emph{Second-order arithmetic} is assumed be the following two-sorted structure:
$$\mathcal A^2=\langle \omega,\baire,ap,+,\times, <,0,1 \rangle,$$
where $\omega$ and $\baire$ are two separate domains and $ap:\baire\times\omega\rightarrow \omega$ is a binary operation of \emph{application}:
$$ap(x,n)=x(n).$$

$+,\times, <,0,1$ have the usual meaning. The underlying language features two sorts of variables:
\begin{enumerate}
 \item $v_1^0, v_2^0,\dots$. ranging over $\omega$ and 
 \item $v_1^1, v_2^1,\dots$. ranging over $\baire$, 
\end{enumerate}
 and two corresponding types of quantifiers: \emph{number quantifiers} $\exists^0,\forall^0,$  and \emph{function quantifiers} $\exists^1,\forall^1$. 
 
\emph{Bounded quantifiers} are those of form $\left(\exists^0v_i^0<n\right)~ \phi$, a shortcut notation for $\exists^0v_i^0~ \left(v_i^0<n ~\land~ \phi\right)$, and $\left(\forall^0v_i^0<n\right)~ \phi$, similarly, a shortcut for $\forall^0v_i^0~ \left(v_i^0<n \implies \phi\right)$.

In this section we identify sets with relations in the usual way, thus for some $k,n\in\omega$ and $A\subseteq \omega^k\times {\left(\baire\right)}^n$ we have $$(m_1,\dots,m_k,f_1,\dots,f_n)\in A \text{~~iff~~} A(m_1,\dots,m_k,f_1,\dots,f_n).$$

We are interested in definability notions related to this structure and all such notions are for subsets of sets of form $\omega^k\times {\left(\baire\right)}^n$. 
\begin{definition}
A set/relation $A\subseteq \omega^k\times {\left(\baire\right)}^n$ is \emph{definable} in $\mathcal A^2$ if there is a formula $\phi[m_1,\dots,m_k,f_1,\dots,f_n]$ such that for any $m_1,\dots, m_k\in\omega$ and any $f_1,\dots,f_n\in\baire,$ $$A(m_1,\dots,m_k,f_1,\dots,f_n)\text{~~iff~~} \mathcal A^2 \models \phi[m_1,\dots,m_k,f_1,\dots,f_n].$$
We say that $A$ is \emph{arithmetical} iff $A$ is definable in $\mathcal A^2$
by a formula without function quantifiers.
A is \emph{analytical} iff $A$ is definable in $\mathcal A^2$.
\end{definition}

Now we can introduce two definability hierarchies that classify relations/subsets according to their quantifier complexity.

\begin{definition}[Arithmetical hierarchy]\label{arithmetical_hd} We define recursively $\ardelta0n, \arsigma0n$ and $\arpi0n$ - three countable sequences of pointclasses.

Let $k,n\in\omega$ and $A\subseteq \omega^k\times {\left(\baire\right)}^n$. Then
\begin{enumerate}
\item  $A\in \ardelta00$ iff $A$ is definable in $\mathcal A^2$
by a formula $\phi[\mathbf w]$ whose only quantifiers are bounded,
\item $A\in \arsigma0n$ iff it is definable by $\exists^0 m_1\forall^0 m_2\dots \mathcal Q^0m_n~\phi[m_1,\dots,m_n,\mathbf w]$ where $\phi$ has only bounded quantifiers and $\mathcal Q^0$ is $\exists^0$ if $n$ is odd and $\forall^0$ otherwise,
\item $ A\in \arpi0n$ iff it is definable by $\forall^0 m_1\exists^0 m_2\dots \mathcal Q^0m_n~\phi[m_1,\dots,m_n,\mathbf w]$ where $\phi$  has only bounded quantifiers and $\mathcal Q^0$ is $\exists^0$ if $n$ is even and $\forall^0$ otherwise, 
\item $\ardelta0n=\arsigma0n~\cap~\arpi0n$.

\end{enumerate}
\end{definition}

Note that existential number quantifiers correspond to countable unions and universal number quantifiers correspond to countable intersections. 

It is easy to see that the above hierarchy classifies all arithmetical sets. Similarly, all analytical sets can be classified with the use of the following (lightface) hierarchy.

\begin{definition}[Analytical hierarchy]\label{analytical_hd} We define recursively $\ardelta1n, \arsigma 1n$ and $\arpi 1n$ - three countable sequences of pointclasses.

Let $k,n\in\omega$ and $A\subseteq \omega^k\times {\left(\baire\right)}^n$. Then
\begin{enumerate}
\item  $\arsigma10=\arsigma01$ and $\arpi10=\arpi01$, 
\item $A\in \arsigma1n$ iff it is definable by $\exists^1 f_1\forall^1 f_2\dots \mathcal Q^1f_n~\phi[f_1,\dots,f_n,\mathbf w]$ where $\phi$ has only number quantifiers and $\mathcal Q^1$ is $\exists^1$ if $n$ is odd and $\forall^1$ otherwise,
\item $ A\in \arpi 1n$ iff it is definable by $\forall^1 f_1\exists^1 f_2\dots \mathcal Q^1f_n~\phi[f_1,\dots,f_n,\mathbf w]$ where $\phi$  has only number quantifiers and $\mathcal Q^1$ is $\exists^1$ if $n$ is even and $\forall^1$ otherwise,

\item $\ardelta1n=\arsigma1n~\cap~\arpi1n$.

\end{enumerate}
\end{definition}

\begin{remark}
\noindent The arithmetical and the analytical hierarchies are called \emph{lightface}, while the Borel and the projective hierarchies are called \emph{boldface}. The lightface notation was the original one used by Kleene to define both lightface hierarchies. The boldface notation for the classical hierarchies was introduced by Addison after he discovered the correspondence between classical and effective notions.
\end{remark}

Both lightface hierarchies were introduced with questions of definability in mind. However, the following simple fact shows the connection to computability theory.

\begin{proposition}
$\ardelta01$ is the set of all computable sets of natural numbers.
\end{proposition}

This fact also partially explains the use of ``effective'' in \emph{effective descriptive set theory}.

Before we state the correspondence discovered by Addison, we need one final notion, \emph{relativization}. 

For $a\in\baire$, \emph{second-order arithmetic in $a$ }is $$\mathcal A^2(a)=\langle \omega,\baire,ap,+,\times, <,0,1,a \rangle,$$
where $a$ is regarded as a binary relation on $\omega$ and all other elements are exactly as in $\mathcal A^2$. Replacing $\mathcal A^2$ by $\mathcal A^2(a)$ in the preceding definitions we get the corresponding relativized notions: $\ardelta0n(a),\arsigma0n(a),\arpi0n(a),\ardelta1n(a),\arsigma1n(a),\arpi1n(a)$ et cetera.

\begin{remark}
The notion of relativization comes from computability theory (the modern name for recursion theory) where the concept of relative computations plays a prominent role. In particular, $\ardelta01(a)$ is the collection of all sets computable in $a$, that is computable with $a$ acting as an \emph{oracle}.
\end{remark}

Finally, here is the correspondence discovered by Addison:
\begin{theorem}
Let $A\subseteq \left(\baire\right)^k$ and let $n>0$. Then:
\begin{enumerate}
\item $A\in \barsigma0n$ iff $A\in\arsigma0n(a)$ for some $a\in\baire$, and similarly for $\barpi0n$.
\item $A\in \barsigma1n$ iff $A\in\arsigma1n(a)$ for some $a\in\baire$, and similarly for $\barpi1n$.
\end{enumerate}
\end{theorem}

\noindent Arguably, this result marked the beginning of what is now called \emph{effective descriptive set theory}.
Effective theory can be seen as a refinement of the classical theory. Apart from clarifying greatly the definability considerations, this development made it possible to use the whole range of effective methods in descriptive set theory. Many classical results have effective proofs that are considerably simpler than known classical proofs, and quite a few important results in descriptive set theory have only effective proofs known.

\section{Invariant Descriptive Set Theory}

Invariant  descriptive set theory is a relatively new area of research within descriptive set theory. This new area is mainly concerned with complexity of equivalence relations and equivalence classes. Historically the subject of equivalence relations entered descriptive set theory through a conjecture of Vaught.

\subsection{Polish groups}

In the process of development of invariant descriptive set theory it became apparent that a particular class of equivalence relations, orbit equivalence relations, is the single most important class of equivalence relations. This is why we start our overview with a subsection on Polish groups, the notion through which orbit equivalence relations are introduced.

\begin{definition}[Polish groups]
  A \emph{topological group} is a group $\left(G,\cdot, 1_G \right)$ with a topology on $G$ such that $(x,y)\mapsto xy^{-1}$ is continuous.

A topological group is \emph{Polish} if its underlying topology is Polish.

A metric $d$ on $G$ is \emph{left-invariant} if $d(gh, gk)=d(h, k)$ for all $g, h, k\in G$.

\end{definition}
  
\noindent The following result characterises Polish subgroups of Polish groups.

\begin{proposition}
Let $G$ be a Polish group and let $H$ be a subgroup of $G$ with the subspace topology.
Then the following are equivalent:
\begin{enumerate}
\item $H$ is Polish,
\item $H$ is $\gdelta$ in $G$, and
\item $H$ is closed in $G$.
\end{enumerate}
\end{proposition}
  
\begin{definition}[Group actions]
  Let $G$ be a group and let $X$ be a set. An action of $G$ on $X$ is a map $$a:G\times X\rightarrow X$$ such that for all $x\in X$ and $g,h\in G$,
\begin{enumerate}
\item $a(1_G,x)=x$,
\item $a(g, a(h, x)) =a(gh, x)$.
\end{enumerate} 
\end{definition}

\noindent Since actions are often canonical, they are not mentioned and $g\cdot x$ is used as a shortcut for $a(g,x).$

\begin{definition}[Orbits and stabilizers]\label{orbits_and_stabilizers_d}
Let a group $G$ act on a set $X$. For each $x\in X$, the \emph{$G-$orbit of $x$}, denoted $[x]_G$ or $G\cdot x$, is the set $\{g\cdot x : g\in G\}$. A subset $A$ of $X$ is \emph{$G-$invariant} if $G\cdot x\subseteq A$ for every $x\in A$.

Furthermore, for every $x\in X$, the \emph{stabilizer of $x$}, denoted $G_x$, is the set $\{g\in G: g\cdot x=x\}$.
\end{definition}

\begin{definition}
Let $G$ be a Polish group, $X$ a standard Borel space, and $a: G\times X\rightarrow X$ an action of $G$ on $X$. If $a$ is a Borel function then we say that $X$ is a (standard) Borel $G-$space.
\end{definition}

The following result characterises stabilizers and orbits induced by Borel actions of Polish groups on standard Borel spaces.

\begin{theorem}
Let $G$ be a Polish group and $X$ a Borel $G-$space. Then for any $x\in X$, $G_x$ is closed and $G\cdot x$ is Borel.
\end{theorem}

\begin{definition}
Let $G$ be a Polish group. A Borel $G-$space $X$ is universal if for any other Borel $G-$space $Y$, there is a Borel embedding of $Y$ into $X$.
\end{definition}

The following useful result has been proved in \cite{Bec96} (Theorem 2.6.6):
\begin{proposition}\label{compact_universal}
For any Polish group $G$, there is a universal Borel $G-$space which is moreover a compact Polish $G-$ space.
\end{proposition}

\subsection{Orbit equivalence relations}

\begin{definition}
Let $X$ be a Polish space. An equivalence relation $E$ is an \emph{orbit equivalence relation} if for some Polish group $G$ acting on $X$, we have $$xEy\iff \exists g\in G~ [g\cdot x=y].$$

If $X$ is a Borel $G-$space, by $E_G^X$ we will denote the orbit equivalence relation induced on $X$ by $G$.
\end{definition}  

\noindent Let's summarize a few relevant facts about the complexity of orbit equivalence relations.

\begin{fact}~
\begin{enumerate}
\item All orbit  equivalence relations are analytic,
\item Not all analytic equivalence relations are orbit equivalence relations,
\item Equivalence classes of orbit equivalence relations are Borel. The converse is not true.
\end{enumerate}
\end{fact}

\begin{example}
$id_X$ is induced by the trivial group acting on $X$.
\end{example}

\begin{example}
Let's define $E_0$, the relation of eventual agreement on $\twon$: $$x~E_0~y\iff \exists m\forall n>m~[x(n)=y(n)].$$ 
This is an orbit equivalence relation,
induced by $\mathbb Z_2^{<\omega}=\oplus_{n\in\omega}\mathbb Z_2^{n}$.
\end{example}

\subsection{Borel reducibility}

Equivalence relations can naturally be viewed as sets, hence the usual machinery of descriptive set theory is readily applicable. However, it turned out that to measure and compare complexities of equivalence relations in a satisfactory way new notions are required. Arguably, the most important of those is Borel reducibility.

\begin{definition}\label{borel_reducibility_d}
Let $E_X$ be an equivalence relation on a Polish space $X$ and let $E_Y$ be an equivalence relation on a Polish space $Y$. 

We say that $E_X$ is \emph{Borel reducible to} $E_Y$ (denoted by $E_X\le_B E_Y$) if $$x~E_X~y\iff f(x)~E_Y~f(y)$$ for some Borel map $f:X\rightarrow Y$.

We write $E\sim_B F$ if $E \le_B F$ and $F \le_B E$.
\end{definition}

The notion of Borel reducibility offers a way of comparing the complexity of equivalence relations: if $E \le_B F$ then $E$ is at most as ``complex'' as $F$ with respect to Borel structures of the underlying Polish spaces. This complexity is often determined by being the most complex object in some natural class. 
\begin{definition}
 Let $\mathcal\ C$ be a class of equivalence relations. We say that $E\in\mathcal C$ is \emph{$\mathcal C-$universal} if $$\forall F\in\mathcal C~~F\le_B\ E.$$
\end{definition}

\begin{theorem}
There is a universal equivalence relation for all orbit equivalence relations
induced by Borel actions of Polish groups.
\end{theorem}

We will see an important example of such a universal orbit equivalence relation in the latter part of the paper.

\begin{definition}
Let $E$ be an equivalence relation on a standard Borel space $X$. We call $E$ \emph{smooth} if $E\le_B id(\twon)$.
\end{definition}

Smoothness is one of notions of simplicity in invariant descriptive set theory. The following result fully characterises smooth relations.
\begin{theorem} Let $E$ be a smooth equivalence relation. Then exactly one of the following holds:
\begin{enumerate}
\item $E\sim_B id(\twon)$;
\item $E\sim_B id(w);$
\item $E\sim_B id(n), $ for some $n\in \omega.$
\end{enumerate} 
\end{theorem}

\subsection{The Urysohn space}\label{urysohn}

\begin{definition}A Polish metric space $X$ is \emph{universal} if for any Polish metric space $Y$ there is an isometric embedding of $Y$ into $X$.

A metric space $X$ is \emph{ultrahomogeneous} if any isometry between finite subspaces of $X$ can be extended to an isometry of $X$. 
\end{definition}

\begin{proposition}[Characterisation of $\mathbb U$]
The universal Urysohn space $\mathbb U$ is the unique, up to an isometry, ultrahomogeneous, universal Polish metric space.
\end{proposition}
  
  In a paper published posthumously in 1927, Urysohn defined a Polish metric space that contained all other Polish spaces as closed subspaces. Its universality and some other nice properties made this space particularly important for DST. The recent surge of interest in the Urysohn space can be traced to 1986 when Kat\v etov in \cite{Kat86} demonstrated a new construction of $\mathbb U$, which allows the extension of every separable metric space to an isometric copy of $\mathbb U$. In this subsection we summarize Kat\v etov's construction.

Another characterization of the Urysohn space relies on the Urysohn property defined below.

\begin{definition}
Let $(X,d)$ be a separable space. We say that a function $f:X\rightarrow \real$ is \emph{admissible} if the following inequality holds for all $x,y\in X:$
\begin{align*}
|f(x)-f(y)|\le d(x,y)\le f(x)+f(y).
\end{align*}
\end{definition}
Every admissible function correspond to a one-point extension of $X$.

\begin{definition}
A metric space $(X, d)$ has the \emph{Urysohn property} if for any finite subset $F$ of $X$ and any admissible function $f$ on $F$, there is $x\in X$ such that $f(p)=d(x, p)$
for any $p\in F$.

\end{definition}

\begin{proposition}
Let $(X, d)$ be a Polish metric space. Then the following statements are equivalent:
\begin{enumerate}
\item $X$ has the Urysohn property,
\item $X$ is universal and ultrahomogeneous.
\end{enumerate}
\end{proposition}

\begin{proposition}
Let $(X, d)$ be a Polish metric space. Then its completion has the Urysohn property too.
\end{proposition}

\noindent The above results show that the completion of any separable metric space that has the Urysohn property is isometric to $\mathbb U$.
The very basic idea of Kat\v etov's construction is to start with a separable metric space $X$, extend it to a space $X_\omega$ with the Urysohn property and then the completion of $X_\omega$ is an isometric copy of $\mathbb U$.

Define $E(X)$ to be the space of all admissible
functions on $X$ endowed with the metric structure induced by the usual sup-metric:
\begin{align*}
d_E(f,g)=\sup\left\{ |f(x)-g(x)|: x\in X \right\}.
\end{align*}

The function $K=x\mapsto d(x,\cdot)$, called \emph{the Kuratowski map}, is an isometric embedding of $X$\ into $E(X)$. $E(X)$ can be seen as an extension of $X.$ From now on we identify $X$\ with its canonical (via the Kuratowski map) embedding into $E(X)$. An important fact about the Kuratowski map is that 

\begin{align}\label{kuratowski_fact_1}
\forall f\in E(X)~~\forall x\in X~~d(f,K(x))=f(x).
\end{align}

\begin{example} Let's prove the above fact.
\begin{proof}
Let $f\in E(X)$ and let $x\in X$. First note that for all $x,y\in X$, \mbox{$f|(x)-f(y)|\le d(x,y)$}, and therefore $|f(y)-d(x,y)|\le f(x)$. Then 
\begin{align*}
d(f,K(x))=\sup\{|f(y)-d(x,y)|:y\in X\}=|f(x)-d(x,x)|=f(x).
\end{align*}
\end{proof}
\end{example}

If $Y\subseteq X$ and $f\in E(Y)$, define $k_X(f):X\rightarrow \real$ (the \emph{Kat\v etov extension} of $f$) by $$k_X(f)(x)=\inf\{f(y)+d(x,y):y\in Y\}.$$ 

The map $f\mapsto k_X(f)$ is an isometric embedding of $E(Y)$ into $E(X)$.

We want to define $X_\omega=\bigcup_{i\in\omega} X_i $ where $X_0=X$ and $X_{i+1}=E(X_i)$. Then $X_\omega$ has the Urysohn property. The problem with this idea is that $E(X)$ is not necessarily separable.
The next example demonstrates this point.

\begin{example}
Consider $\omega$ equipped with the trivial metric 
\begin{align*}
d(x,y)=\begin{cases} 1,& \text{ if } x\neq y.\\
0, & \text{ otherwise.} \end{cases}
\end{align*}
Then $E(\omega)$\ is not separable. 
\begin{proof}
To show this, we will exhibit an uncountable discrete subspace of  $E(\omega)$. Fix a bijection $f:\omega\rightarrow\omega^2.$ For every $r\in\baire$ define 
\begin{align*}
f_r(x)=\begin{cases}
2,& \text{ if }f(x)(1)=r(f(x)(2)),\\
1,& \text{ otherwise.}
\end{cases}
\end{align*}

It is easy to see that all such functions are admissible. Moreover, for any distinct $r_1,r_2\in\baire$, $d_E(f_{r_1},f_{r_2})=1$, hence the set of all such functions forms an uncountable discrete subspace of $E(\omega)$.
\end{proof}
\end{example}

To get separability we need to introduce the concept of support for admissible functions.
\begin{definition}
Let $f$ be an admissible function on $(X,d)$ and let $S\subseteq X$. We say that $S$\ is a \emph{support} for $f$ if for all $x\in X$,
$$f(x)=\inf \{f(y)+d(x,y): y\in S\}.$$
$f$ is called \emph{finitely supported} if there is a finite support for $f$.
\end{definition}

\noindent Define $$E(X,\omega)=\{f\in E(X):f \text{ is finitely supported}\}.$$

Then if $X$\ is separable, $E(X,\omega)$ is separable too. Now we can define inductively 
\begin{align*}
X_0=X,\\
X_{i+1}=E(X_i,\omega),\\
X_\omega=\bigcup_{i\in\omega}X_i,
\end{align*}

\noindent and let $d_{X_\omega}$ be the canonical extension of $d$ on $X_\omega$. Then if $X$ is separable, $X_\omega$ is separable too. Let's show $X_\omega$ has the Urysohn property.
\begin{proof}
Let $f\in E(X_\omega)$ and let $F\subseteq X_\omega$ be finite. Since $F$ is finite, it must be that $F\subseteq X_m$ for some $m\in\omega$. 

Define $f_F=f|_F$. Then $f_F\in E(F)\subseteq E(X_m,\omega)=X_{m+1}\subseteq X_\omega$ and using (\ref{kuratowski_fact_1}) we get that $$\forall x\in X_m~~d(f_F,K(x))=d(f_F,x)=f_F(x).$$ And, finally,
$$\forall x\in F~~d(f_F,x)=f(x).$$
\end{proof}

To finish the construction, note that since $X_\omega$ has the Urysohn property,  the completion of it is an isometric copy of the Urysohn space. 

\begin{definition}
From now on we fix the Urysohn space $\mathbb U$ as the completion of $(\real_\omega,d_{\real_\omega})$.
\end{definition}

The following proposition states an important property of the Urysohn space that will be used by us later on.

\begin{proposition}\label{compact_isometries} If $K,L\subset \mathbb U$ are compact and $\phi:K\rightarrow L$ is an isometry, then there is an isometry $\hat \phi:\mathbb U\rightarrow \mathbb U$ with $\hat\phi|_K=\phi$.
\end{proposition}

\begin{remark}
Note that $F(\mathbb U)$ can naturally be seen as a Borel space of all Polish metric spaces.
\end{remark}

\begin{example}
Let's show that if $D$ is dense in $X$, then $k_X\left(E(D,\omega)\right)$ is dense in $E(X,\omega)$.
\begin{proof}

Let $g\in E(D,\omega)$ and let $\{x_1,\dots,x_k\}\subseteq X$ be a finite support for $g$. For $1\le i\le k$, let $\seq{d_{i.}}n$ be such sequences in $D$ that $d_{i,n}\rightarrow x_i$. Define $$g_n=k_D\left( \prefix g{\{d_{1,n},\dots, d_{k,n}\}}\right).$$

\noindent Then $g_n\in k_D\left(E(\{d_{1,n},\dots, d_{k,n}\})\right)$, and since $k_D\left(E(\{d_{1,n},\dots, d_{k,n}\})\right)\subseteq E(D,\omega),$ $g_n\in E(D,\omega)$. Finally, note that $k_X(g_n)\rightarrow g$.
\end{proof}
\end{example}

\begin{remark}
The question of identifying the Polish metric spaces $X$ such that $E(X)$ is separable is of independent interest. Jullien Melleray (\cite{Mel07}) proved the following results:
\begin{enumerate}
\item If $X$ is Polish and $E(X)$ is separable, then $X$ is Heine-Borel.
\item Each of the following two conditions is equivalent to $E(X)$\ being separable:
\begin{enumerate}
\item $E(X)=\overline{E(X,\omega)}$,
\item $X$ has the collinearity property.
\end{enumerate}
\end{enumerate}

The collinearity property is defined in the following way. Let $(X,d)$ be a metric space and let $\epsilon>0$. An ordered triple of points $(x_1,x_2,x_3)$ are said to be \emph{$\epsilon-$collinear} if $d(x_1,x_2)\le d(x_1,x_2)+d(x_2,x_3)-\epsilon.$ Then we say $X$ has the \emph{collinearity property} if for every infinite $A\subseteq X$ and every $\epsilon>0$ there is a triple of distinct points belonging to $A$ that are $\epsilon-$collinear. 

\end{remark}

\section{An application of invariant DST to the problem of classification of Polish metric spaces}

This section summarizes and explains the main result of the section 14 of \mbox{Su Gao's} book (\cite{Gao09}), that the isometry classification problem of Polish metric spaces is a universal orbit equivalence relation.

\subsection{The space of all Polish metric spaces}\label{space_of_polish}

We are interested in studying complexities of equivalence relations defined either on the class of all Polish metric spaces or on some proper subclass of it. Examples of such equivalence relations include:
\begin{enumerate}
\item homeomorphic classification of Polish metric spaces.
\item isometric classification of Polish metric spaces, denoted by $\cong_i$,
\item isometric classification of compact Polish metric spaces, etc.
\end{enumerate}

\noindent From now on, we denote the collection of all Polish metric spaces by $\mathcal X$. 
\noindent To study an equivalence relation using invariant descriptive set theory and its methods, the equivalence relation must be defined on a space with a Borel structure, but there is no natural Borel structure defined on $\mathcal X$. To overcome this difficulty, we need to encode elements of $\mathcal X$ as elements of some known space with a Borel structure.
We will use two such encodings:

\begin{enumerate}
\item
The idea behind the first encoding is that the metric structure of a Polish metric space $X$\ can be ``recovered'' from distances between any fixed countable dense subset of $X$. Then a Polish metric space can be encoded as an element of $\real^{\omega\times \omega}$, a space of double sequences of real numbers. 
\item We know that the Effros Borel space of $F(\mathbb U)$ can be seen as a Borel space of all Polish metric spaces.
In this approach, Polish metric spaces are encoded as closed subspaces of $\mathbb U$.\end{enumerate}

Now we will discuss both of these encodings in detail and then will show that both are equivalent in some important way.

\subsection{Encoding Polish metric spaces as elements in $\real^{\omega\times\omega}$}\label{encoding1}

A metric structure of a Polish metric space can be recovered from the metric structure of any of its countable dense subspaces. And a metric structure of a countable space can be encoded as a double sequence of real numbers. However, not every double sequence of real numbers correspond to some countable metric space.

Define $\mathbb X$ to be the subspace of $\real^{\omega\times\omega}$ consisting of elements $\seq{r}{i,j}$ such that for all $i,j,k\in \omega$:
\begin{enumerate}
\item $r_{i,j}\ge 0$ and $r_{i,j}=0$ iff $i=j;$
\item $r_{i,j}=r_{j,i};$ 
\item $r_{i,j}\le r_{i,k}+r_{k,j}.$
\end{enumerate}

With the three above restrictions, 
every element of $\mathbb X$ correspond to some dense subset in some Polish metric space, and to every enumeration of every dense subset of every Polish metric space there is a corresponding element of $\mathbb X$. To get a satisfying correspondence between $\mathcal X$ and $\mathbb X$ we need to fix a particular enumeration of a particular countable dense subset for every Polish metric space. This is possible with the use of the Axiom of Choice.
Thus, for every Polish space $X$ we fix $r_X$, an element of $\mathbb X$, in such a way that $r_X$ represents a metric structure of some dense subset of $X$. In the other direction. for any $r\in\mathbb X$, we define a metric space $X_r$ to be the completion of the metric space $(\omega, d_{X,\omega})$ where $d_{X,\omega}(i,j)=r_{i,j}$.

\noindent Note that $\mathbb X$ is Polish (via \ref{polish_basic}(2)), as it is a closed subset of $\real^{\omega\times\omega}$ which is Polish too.

\subsection{Encoding Polish metric spaces as elements of $F(\mathbb U)$}\label{encoding2}

We already mentioned that, in an informal way, $F(\mathbb U)$ can be seen as a standard Borel space of all Polish metric spaces.  We identify Polish metric spaces with closed subspaces of $\mathbb U$ in the following way. Given a Polish metric space $X$, we construct a space with the Urysohn property $X_\omega$ and its completion - $\overline{X_\omega}$. Via the Kuratowski map, $X$ is isometric to a closed subspace of $\overline{X_\omega}$, and $\overline{X_\omega}$ itself is isometric to $\mathbb U$. Combining those two isometries, gives us an isometry between $X$ and a closed subspace of $\mathbb U$. Thus, we have a mapping between $\mathcal X$ and $F(\mathbb U)$.

Moreover, this construction can be done in a canonical, that is choice-free way and the following results can be established  with respect to the discussed correspondences between $\mathcal X, \mathbb X$ and $F(\mathbb U)$.

\begin{proposition}\label{prop221}
There is a Borel embedding $J$ from $\mathbb X$ into $F(\mathbb U)$ such that for any $r\in \mathbb X$, $X_r$ is isometric with $J(r)$. Furthermore, an isometry between $X_r$ and $J(r)$ is extensible to an isometry between $\overline{\left(X_r\right)_w}$ and $\mathbb U$. 
\end{proposition}

\begin{proposition}
There is a Borel embedding $j$ from $F(\mathbb U)$ into $J(\mathbb X)$ such that for any $F\in F(\mathbb U)$, $F$ is isometric with $j(F)$.
\end{proposition}

\begin{proposition} 
There is a Borel isomorphism $\Theta$ between $J(\mathbb X)$ and $F(\mathbb U)$ such that for any $r\in \mathbb X$, $X_r$ is isometric with $\Theta(r)$.
\end{proposition}

\subsection{Classification problems and invariant descriptive set theory}

A \emph{classification problem} is associated with a class of mathematical structures and a notion of equivalence and is usually formulated as ``What are the objects of a given type, up to some equivalence?''. Generally, in mathematics a \emph{classification theorem} answers the classification problem.
Examples of such results include:

\begin{enumerate}
\item Classification of Euclidean plane isometries,
\item Classification of finite simple groups, 
\item  The Artin-Wedderburn theorem is a classification theorem for semisimple rings and semisimple algebras.
\end{enumerate}

\noindent In the context of invariant descriptive set theory, classification theorems are concerned with the complexity of respective classification problems seen as equivalence relations on some suitable standard Borel spaces. Such results usually require two parts:
\begin{enumerate}
\item[(1)] determination of complete invariants for the objects in question,
\begin{itemize}\item[-] this corresponds to a Borel
reduction to a known equivalence relation,\end{itemize}
\item[(2)] should leave no room for significant improvement, \begin{itemize}\item[-] this can be interpreted as a natural equivalence relation that is Borel bireducible
with the classification problem.\end{itemize}
\end{enumerate}

Part $(1)$ is called a \emph{completeness} result, part $(2)$ can be seen as a \emph{complexity determination} result.

We are interested in classification results for Polish metric spaces. In this case a classification problem is an equivalence relation defined either on $\mathcal X$, equipped with a Borel structure through one of the mentioned encodings, or on a subspace of $\mathcal X$ (for instance the space of compact Polish metric spaces). Our main focus is to examine the isometric classification of Polish metric spaces, but we will start with a discussion of simpler results concerning Polish metric spaces. 

Let $\mathcal X_c$ denote the hyperspace of all compact Polish metric spaces. If $\mathcal X_c$, as a subset of $\mathcal X$, is Borel, then $\mathcal X_c$, as a subspace of $\mathcal X$, is a standard Borel space. Let's show this.
\begin{example}\label{compact_example} Define $$\mathbb X_c=\{r\in \mathbb X: X_r \text{ is compact}\}.$$
Let's show that $\mathbb X_c$ is a Borel subset of $\mathbb X$.
\begin{proof}
We know that a complete metric space is compact if and only if it is totally bounded. Then 
\begin{align*}
\mathbb X_c=\bigcap_{\epsilon\in\mathbb Q}\bigcup_{k\in \omega}\bigcap_{j\in\omega}\bigcup_{n\le k}\{x\in\mathbb X^{\omega\times\omega}:x_{j,n}\le \epsilon\}.
\end{align*}
It follows that $\mathbb X_c$ is Borel.
\end{proof}
\end{example}

The following classification theorem characterises the isometric classification problem of Polish metric spaces. 

\begin{theorem}\label{t_smooth}
The isometric classification problem for compact Polish metric spaces is smooth.
\end{theorem}

\begin{remark}
Note that smoothness of an equivalence relation implies that the relation is simple enough so that its invariants can be represented as elements of a standard Borel space: if an equivalence relation $E$ on $X$ is smooth, then there is Borel map $c:X\rightarrow Y$ where $Y$ is a standard Borel space such that $x E y\iff c(x)=c(y).$
\end{remark}

The previous theorem shows that the isometric classification of 

\subsection{Isometric classification of Polish metric spaces}

Define $E_I$ as the orbit equivalence relation on $F(\mathbb U)$ induced by $\text{Iso}(\mathbb U)$, the group of isometries on the Urysohn space.
So, for $\boldsymbol c_1,\boldsymbol c_2 \in F(\mathbb U)$,
$$\boldsymbol c_1~E_I~\boldsymbol c_2\implies \boldsymbol c_1\text{ is isometric to } \boldsymbol c_2.$$
The implication in the other direction is not true in general, however, via \ref{compact_isometries}, it holds for compact subsets of the Urysohn space.

This subsection is devoted to proving the following result, proven independently by Clemens (\cite{Cle01}) and by Gao and Kechris, that fully classifies Polish metric spaces up to an isometry.

\begin{theorem}\label{t_full}~
\begin{enumerate}
\item (Completeness) $\cong_i~\le_B E_I,$
\item (Complexity) $\cong_i~\sim_B E_I$ and $E_I$ is a universal orbit equivalence relation.

\end{enumerate}
\end{theorem}

\begin{proof}[Proof of (Completeness)]

We identify $\mathcal X$ with $\mathbb X$ is this proof. We need to find a Borel function \mbox{$f:\mathbb X\rightarrow \mathbb U$} such that for all $X,Y\in\mathcal X,$ $$X\cong_i Y\iff f(r_X)~E_I~ f(r_Y).$$ Let's check that $J$, defined in Proposition \ref{prop221}, satisfies this condition:
\begin{enumerate}
\item[($\Leftarrow$)] Suppose $J(r_X)~E_I~ J(r_Y)$. Then $J(r_X)\cong_i J(r_Y),$ and by \ref{prop221} and from the transitivity of $\cong_i$ we get $X\cong_i J(r_X)\cong_i J(r_Y)\cong_i Y.$
\item[($\Rightarrow$)]  Suppose $X\cong_i Y$. From \ref{prop221} we have two isometries, $\phi_X:\overline{X_\omega}\rightarrow \mathbb U$ and $\phi_Y:\overline{Y_\omega}\rightarrow \mathbb U$, such that $\phi_X(X)=J(r_X)$ and $\phi_Y(Y)=J(r_Y)$. If $\pi$ is an isometry between $X$ and $Y$, it can be extended to $\pi^*$, an isometry between $\overline{X_\omega}$ and $\overline{Y_\omega}$.  Then $\phi_X \circ \pi^*\circ \phi_Y^{-1}$ is an isometry on $\mathbb U$. Moreover, $$\phi_X \circ \pi^*\circ \phi_Y^{-1}(J(r_Y))=J(r_X).$$ It follows that $J(r_X)~E_I~ J(r_Y)$.
\end{enumerate}
\end{proof}

\begin{proof}[Proof of (Complexity)]
Again, we identify $\mathcal X$ with $\mathbb X$ is this proof.
For this part of the theorem, it is sufficient to show that $E_I$ is a universal orbit equivalence relation and the required bireducibility will follow as a simple consequence. 

Let $G$ be a Polish group and let $X$ be a Borel $G$-space. To show the universality of $E_I$,  we will construct a Borel function $x\mapsto M_x$ where $x\in X$ and $M_x$ is a Polish metric space (identified with some element $r_{M_x}\in\mathbb X$), such that $x~E_G^X~y\iff M_x\cong_i M_y.$ We may assume that $X$ is compact and the reason is following. Due to $\ref{compact_universal}$, we know that there is a compact universal Borel $G-$space, let's call it $U_c$. It is easy to see that $E_G^X\le_B E_G^{U_c}$, hence it is sufficient to show that $E_G^{U_c}\le_B E_I$. 

First, we need the following lemma.
\begin{lemma}\label{lemma228}
Let $G$ be a Polish group and $X$ a compact Polish $G-$space. Let $d_X\le1$ be a compatible metric on $X$. Then there is a left-invariant compatible metric $d_G\le1$ on $G$ such that, for any $x\in X$ and $g, h\in G$,
$$d_G(g,h)\ge \frac12 d_X(g^{-1}\cdot x, h^{-1}\cdot x).$$
\end{lemma}

Let $d_X$ be a compatible metric on $X$ with $d\le 1.$ By Lemma \ref{lemma228}, there is a compatible metric $d_G$ on $G$ with $d_G\le1$. Without loss of generality we assume that there are exist $x,y\in X$ with $d_X(x,y)=1$ and that $$\sup\{d_G(g,h): g,h\in G\}=1.$$

By left invariance of $d_G$, it follows that $\sup\{d_G(g,h): g,h\in G\}=1$ for any $g\in G$.

Let $x\in X$ and let's define $M_x$. First, fix a countable dense subset $D$ of $X$ with $\seq xn$ being its enumeration. Fix a bijection $\pi: \mathbb Z\rightarrow \omega$ defined as $$\pi(n)=\begin{cases}2n,&\text{if } n\ge0,\\-2n-1,&\text{otherwise.}\end{cases}$$ Let $H=G\times \mathbb Z\times \{0,1\}$ and define the following metric on $H$:

\begin{align*}
d_x\left((g_1,n_1,i_1),(g_2,n_2,i_2)\right)=
\begin{cases}
d_G(g_1,g_2),& \text{if $n_1=n_2$ and $i_1=i_2,$}\\
\frac32+ \frac{1+d_G(g_1,g_2)}{4^{|n_1-n_2|}},& \text{if $n_1\neq n_2$ and $i_1=i_2$,}\\
1+\frac{1+d_X(x_{\pi(n_1-n_2)}, g_2^{-1}\cdot x)}{4^{\pi(n_1-n_2)+1}},& \text{if $i_1=0$ and $i_2=1,$}\\
1+\frac{1+d_X(x_{\pi(n_2-n_1)}, g_1^{-1}\cdot x)}{4^{\pi(n_2-n_1)+1}},& \text{if $i_1=1$ and $i_2=0.$}
\end{cases}
\end{align*}

Let $\hat G$ be the completion of $G$ with $d_G$ and let $\hat H$ be the completion of $H$ with $d_X$. $\hat H$ can be seen as  a countable union of copies of $\hat G$ with the completed metric $d_X$ on each copy coinciding with the completed metric $d_G$. Define $M_x$ to be $\hat H$. It is easy to see that $x\mapsto M_x$ is a Borel function, since a countable dense subset of $M_x$ can be obtained canonically from a canonical countable dense subset of $G$, and the distances between elements of this subset are defined in terms of $d_X$, which is itself a Borel function.

What is left to prove is that for all $x,y\in X,~~ xE_G^Xy\iff M_x\cong_i M_y:$

\begin{enumerate}
\item[($\Rightarrow$)] Let $x,y\in X$ and suppose $xE_G^Xy$ with $y=h\cdot x$. Then the map \mbox{$(g,n,i)\mapsto (hg,n,i)$} is an isometry between $(H,d_x)$ and $(H,d_y)$. This map extends uniquely to an isometry between $M_x$ and $M_y$ and thus we have $M_x\cong_i M_y$. The map is clearly a bijection, hence the only thing to check is whether it preserves the metric. The definition of $d_x$ has four cases. The metric preservation for the first two cases easily follows from the left-invariance of $d_G$. The other two cases are symmetrical, hence it is sufficient to verify only one of them. Let $(g_1,n_1,i_1),(g_2,n_2,i_2)\in H$ with $i_1=0$ and $i_2=1$. Then

\begin{align*}
d_x\left((g_1,n_1,i_1),(g_2,n_2,i_2)\right)=1+\frac{1+d_X(x_{\pi(n_1-n_2)}, g_2^{-1}\cdot x)}{4^{\pi(n_1-n_2)+1}}=\\
1+\frac{1+d_X(x_{\pi(n_1-n_2)}, (hg_2)^{-1}\cdot x)}{4^{\pi(n_1-n_2)+1}}=d_x\left((hg_1,n_1,i_1),(hg_2,n_2,i_2)\right).
\end{align*}
 
\item[($\Leftarrow$)] Let $x,y\in X$ and suppose $M_x\cong_i M_y.$ Let $\phi$ be an isomorphism between $M_x$ and $M_y$. We need to show that $y=h\cdot x$ for some $h\in G$.

Note that $M_x$ is such a countable union of copies of $\hat G$ that each copy has the diameter $\le1$ and distances between elements of distinct copies are greater than $1$. This means that $\phi$ sends distinct copies of $\hat G$ in $M_x$ into distinct copies of $\hat G$ in $M_y$. For all $n\in\mathbb Z,i\in\{0,1\}$ define \mbox{$\hat H_{n,i}=\hat G\times \{n\}\times\{i\}$}. Then  $\phi$ induces a bijection \mbox{$f:\mathbb Z\times\{0,1\}\rightarrow \mathbb Z\times\{0,1\}$}, such that for all $n\in\mathbb Z,i\in\{0,1\}$,
$$\phi\left(\hat H_{n,i}\right)=\hat H_{f(n,i)}.$$

We can think of copies of $\hat G$\ as organized in two countable sequences, $\hat H_{n,0}$ and $\hat H_{n,1}$. Note, that distances between elements from different sequences are greater than $\frac32$, while distances between elements from the same sequence are at most $\frac32$.  Thus, $\phi$ sends copies from one sequence to the same sequence. Finally, note that for all $i\in\{0,1\}$, $n_1,n_2\in\mathbb Z,h_1,h_2\in \hat G$ with $n_1\neq n_2,$ we have
$$d_x\left((h_1,n_1,i),(h_2,n_2,i) \right)\le\frac32 +\frac2{4^{|n_1-n_2|}}.$$

The above inequality implies that for any $i\in\{0,1\}$ and $m\in\mathbb Z$, if \mbox{$f(0,i)=(n_0,j)$} and $f(m,i)=(n_0+m_0,j)$, then $|m_0|\le |m|$. By induction, from the bijectivity of $f$, we have a stronger relation:\ $|m_0|=|m|$.
\begin{enumerate}
\item Suppose $f(0,0)=(n_0,0)$ for some $n_0\in \mathbb Z$. Then $f(0,1)=(m_0,1)$ for some $m_0\in \mathbb Z$. Let's show $n_0=m_0$. For suppose otherwise and let $g_1,g_2\in G$, then we will get a contradiction:

\begin{align*}
d_x\left((g_1,0,0),(g_2,0,1)\right)=1+\frac{1+d_X(x_{\pi(0)}, g_2^{-1}\cdot x)}{4^{\pi(0)+1}}\ge \frac 54,
\end{align*}

but 

\begin{align*}
d_y\left((g_1,n_0,0),(g_2,m_0,1)\right)=1+\frac{1+d_Y(x_{\pi(n_1-n_2)}, g_2^{-1}\cdot y)}{4^{\pi(n_1-n_2)+1}}\le 1+\frac2{4^2} =\frac 98.
\end{align*}

Let $m\in \mathbb Z$. We know that $f(m,0)=(n_0\pm m,0), $ let's show that $f(m,0)=(n_0+ m,0)$. For $g_1,g_2\in G$ we have 
\begin{align*}
d_x\left((g_1,m,0),(g_2,0,1)\right)\in\left[1+\frac1{4^{\pi(m)+1}}, 1+\frac2{4^{\pi(m)+1}}\right],\\
d_y\left((g_1,n_0+m,0),(g_2,n_0,1)\right)\in\left[1+\frac1{4^{\pi(m)+1}}, 1+\frac2{4^{\pi(m)+1}}\right],\\
d_y\left((g_1,n_0-m,0),(g_2,n_0,1)\right)\in\left[1+\frac1{4^{\pi(-m)+1}}, 1+\frac2{4^{\pi(-m)+1}}\right].
\end{align*}

Note that the third interval is disjoint from the first two, which are the same. It follows that for all $m\in\mathbb Z,$ $f(m,0)=(n_0+ m,0)$. 

By $\ref{chouqet_polish}$ we know that $G$ is comeager in $\hat G$. Then there is a comeager subset $C$ of $G$ such that for all $g\in G$ and $m\in \mathbb Z$, $\phi(g,m,0)\in H$ and $\phi(g,0,1)\in H$.
Let $g_1,g_2\in C$ and let $m \in \mathbb Z$, then
\begin{align}\label{lll1}
d_x\left((g_1,m,0),(g_2,0,1)\right)=1+\frac{1+d_X(x_{\pi(m)}, g_2^{-1}\cdot x)}{4^{\pi(m)+1}}.
\end{align}

Let $h_{1,m}\in G$ be such that $\phi(g_1,m,0)=(h_{1,m},n_0+m,0)$ and let $h_{2}\in G$ be such that $\phi(g_2,0,1)=(h_{2},n_0,1)$. We have 
\begin{align}\label{lll2}
\begin{split}
d_x\left((g_1,m,0),(g_2,0,1)\right)=d_y\left(\phi(g_1,m,0),\phi(g_2,0,1)\right)=\\
d_y\left((h_{1,m},n_0+m,0),(h_{2},n_0,1)\right)=1+\frac{1+d_X(x_{\pi(m)}, h_2^{-1}\cdot x)}{4^{\pi(m)+1}}.
\end{split}
\end{align}

Comparing $(\ref{lll1})$ with $(\ref{lll2})$ we get that for all $m\in\mathbb Z$ we have $d_X\left(x_{\pi(m)}, g_2^{-1}\cdot x\right)=d_X\left(x_{\pi(m)}, h_2^{-1}\cdot y\right)$. And finally, $y = h_2g_2^{-1}\cdot x$.

\item Suppose $f(0,0)=(n_0,1)$ for some $n_0\in \mathbb Z$. An analogous argument shows that $f(0,1)=(n_0,0)$ and $f(m,1)=(n_0-m,0)$ for all $m\in \mathbb Z$. The rest of the argument is analogous, with the necessary modifications.
\end{enumerate}
\end{enumerate}

\end{proof}

\begin{remark} While Proposition \ref{t_smooth} shows that the isometric classification of compact Polish metric spaces is very simple, the above result states that the same classification for Polish metric spaces is the most complex one (in the natural class of all orbit equivalence relations).
\end{remark}

\subsection{Other classification problems for Polish metric spaces}

In this subsection we will describe some other results and examples concerning complexities of equivalence relations on the space of all Polish metric spaces. 
\begin{definition}
Let  $X$ and $Y$ be metric spaces.

We say that $X$ and $Y$ are \emph{uniformly homeomorphic} if there is a bijection $f:X\rightarrow Y$ such that both $f$ and $f^{-1}$ are uniformly continuous.

We say that  $X$ and $Y$ are \emph{Lipschitz isomorphic} if there is a bijection \mbox{$f: X\rightarrow Y$} such that both $f$ and $f^{-1}$ are Lipschitz.

We say that  $X$ and $Y$ are \emph{isometrically biembeddable } if there are isometric embeddings $f: X\rightarrow Y$ and $g: Y\rightarrow X$.

\end{definition}

\begin{example} Let's show that uniform homeomorphism is a $\barsigma11$ relation on $\mathcal X$.

\begin{proof}
We know that an analytic set is a continuous image of a Polish space. Combining this with the Proposition \ref{borel_to_continuous}, we have another characterisation of analytic sets:\ a set is analytic if it is a Borel image of a Polish space. This means that a set $B$ is analytic if there is a Polish space $X$ and a Borel function $f$ with 
$$y\in B\iff \exists x\in X~~y=f(x),$$ or, equivalently, 
$$y\in B\iff \exists x\in X~~\text{some Borel conditions}.$$

We identify $\mathcal X$ with $\mathbb X$. Let $\mathcal A\subseteq \mathbb X^2$ be the uniform homeomorphism relation on Polish metric spaces. Then $(r_X,r_Y)\in \mathcal A$ if 
\begin{enumerate}
\item [(*)] there exists $r_Z$ with $r_X\cong_i r_Z$ and
\item[(**)] $f:(\omega, d_{Z,\omega})\rightarrow (\omega, d_{Y,\omega})$ defined as the identity on $\omega$, is a uniformly continuous bijection with $f^{-1}$ also uniformly continuous.
\end{enumerate}

To show that $\mathcal A$ is analytic, it is sufficient to demonstrate that both (*) and (**) are Borel conditions. (*) is Borel since $\cong_i$ is an orbit equivalence relation and all of its equivalence classes are Borel. Example \ref{compact_example} shows that compactness of Polish metric spaces is a Borel condition. By a similar argument, uniform continuity is a Borel condition and hence (**) is also a Borel condition. Thus $\mathcal A$ is analytic.
\end{proof}
\end{example}

\begin{fact} Lipschitz isomorphism and isometric biembeddability are analytic equivalence relations.
\end{fact}

Another important equivalence relation on Polish metric spaces is \emph{homeomorphism} relation. Let $\cong_h$ denote the homeomorphism relation on Polish metric spaces and let $\cong_{ch}$ denote the restriction of $\cong_h$ to compact spaces. The following facts are known with respect to complexity of these relations.
\begin{fact}\label{ff1} ~
\begin{enumerate}
\item $\cong_h$ is $\barsigma12$;
\item $\cong_{ch}$ is analytic.
\end{enumerate} 
\end{fact}

\begin{remark} \ref{ff1}(1) is an upper bound; it is an open question whether $\cong_h$ is simpler than $\barsigma12$.

\ref{ff1}(2) can be easily proven by an argument similar to the one used in the previous example. Let $X,Y$ be compact Polish metric spaces. Then $X\cong_{ch} Y$ if there exists a continuous bijection $f:X\rightarrow Y$. Bijectivity and uniform continuity on Polish metric spaces are Borel conditions, hence $\cong_{ch}$ is analytic. 
\end{remark}




\begin{thebibliography}{99}


\bibitem{Bec96} Becker H., Kechris A. S., \textit{The descriptive set theory of Polish group actions,} London Mathematical Society Lecture Note Series, vol. 232, Cambridge University Press, 1996.


\bibitem{Cle01} Clemens J. D., Gao S., Kechris A. S., \textit{Polish metric spaces: their classification and isometry groups,} Bulletin of Symbolic Logic 7 (2001), no. 3, 361-375.


\bibitem{Gao09} Gao S., \textit{Invariant Descriptive Set Theory,} Pure and Applied Mathematics (Boca Raton), vol. 293, CRC Press, Boca Raton, FL, 2009.


\bibitem{God38} G\"odel K., \textit{The consistency of the Axiom of Choice and of the Generalized Continuum-Hypothesis,} PNAS 24 (1938), 556-557. 

\bibitem{Jec00} Jech T., \textit{Set Theory,} Springer Monographs in Mathematics, Springer-Verlag, 2006.

\bibitem{Kan95} Kanamori A., \textit{The emergence of descriptive set theory,} Synthese (1995), 251: 241-262.

\bibitem{Kan_HI} Kanamori A., \textit{The higher infinite: large cardinals in set theory from their beginnings,} Perspectives in Mathematical Logic, Springer-Verlag, 1994.

\bibitem{Kat86} Kat\v etov M., \textit{On universal metric spaces,} in \textit{General Topology and Its Relations to Modern Analysis and Algebra, IV} (Prague, 1986), Heldermann, Berlin, 1988, 323-330.

\bibitem{Ker94} Kechris A. S., \textit{Classical Descriptive Set Theory.} Graduate Texts in Mathematics, vol. 156, Springer-Verlag, New York, 1995.

\bibitem{Has94} Ki H., Linton T., \textit{Normal numbers and subsets of $\mathbb N$ with given densities,} Fund. Math. (2) 144(1994), 163-179

\bibitem{Mil00} Miller A. W., Popvassiliev S. G., \textit{Vitali sets and Hamel bases that are Marczewski measurable}, Fund. Math. 166 (2000), 269-279.

\bibitem{Mos94} Moschovakis Y. N., \textit{Descriptive Set Theory.} Mathematical Surveys and Monographs, vol. 155, American Mathematical Society, 2009.

\bibitem{Mel07} Melleray J., \textit{On the geometry of Urysohn's universal metric space}, Topology and its Applications 154(2007), 384-403.


\bibitem{Myc64} Mycielski J., Swierszczkowski S., \textit{On the Lebesgue measurability and the axiom of determinateness}, Fund. Math. 54 (1964), 67-71.


\bibitem{Sri98} Srivastava S. M., \textit{A course on Borel sets,} Graduate Texts in Mathematics, vol. 180, Springer-Verlag, New York, 1998.


\end{thebibliography}
\end{document}